\documentclass[reqno]{amsart}

\usepackage{amsmath,amssymb,graphicx
}
\usepackage{enumitem}

\usepackage{youngtab}
\usepackage[latin1]{inputenc}
\usepackage[colorlinks,citecolor=blue]{hyperref}
\newtheorem{Theorem}{Theorem}[section]
\newtheorem{proposition}[Theorem]{Proposition}
\newtheorem{corollary}[Theorem]{Corollary}
\newtheorem{lemma}[Theorem]{Lemma}
\theoremstyle{definition}
\newtheorem{definition}[Theorem]{Definition}
\newtheorem{Question}[Theorem]{Question}
\newtheorem{example}[Theorem]{Example}

\theoremstyle{remark}
\newtheorem{Remark}[Theorem]{Remark}

\newcommand{\Pp} {\mathcal{P}}

\newcommand{\Inv} {{\rm Inv}}
\newcommand{\inv}{{\rm inv}}

\newcommand{\A} {\mathcal{A}}
\newcommand{\B} {\mathcal{B}}

\newcommand{\G}{\mathcal{G}}
\newcommand{\U}{\mathcal{U}}
\newcommand{\SSF}{{\rm SSF}}

\allowdisplaybreaks

\title[A new family of posets]{A new family of posets generalizing the weak order on some Coxeter groups}
\author{Fran\c{c}ois Viard}
\date{\today}

\begin{document}
\maketitle

\begin{abstract}
We construct a poset from a simple acyclic digraph together with a valuation on its vertices, and we compute the values of its M\"obius function. We show that the weak order on Coxeter groups $A_{n-1}$, $B_n$, $\widetilde{A}_n$, and the flag weak order on the wreath product $\mathbb{Z} _r \wr S_n$ introduced by Adin, Brenti and Roichman, are special instances of our construction. We conclude by associating a quasi-symmetric function to each element of these posets. In the $A_{n-1}$ and $\widetilde{A}_n$ cases, this function coincides respectively with the classical Stanley symmetric function, and with Lam's affine generalization.
\end{abstract}

\section{Introduction}

The weak order on a Coxeter group $W$ is a partial order on $W$ which plays a significant role in many areas of algebra and algebraic combinatorics as Grassmannian geometry and Schubert calculus (see \cite{FGR}). Moreover, it is closely related to the geometry of the root system associated with a Coxeter group (see \cite{Dy11}, \cite{HD}, \cite{Pil}, or \cite{DHR}), and to the theory of quasi-symmetric functions (see \cite{BM} for a general survey) thanks to the Stanley symmetric functions. These functions were introduced by Stanley in \cite{S1}, in order to enumerate the reduced decompositions of any permutation $\sigma$ in the symmetric group $S_n$, equivalently enumerating the maximal chains from the identity to $\sigma$ in the weak order on $S_n$, and turned out to be of fundamental importance in many areas of algebra (see \cite{BJS}). In \cite{Lam}, Lam generalized Stanley's work to the affine Coxeter group of type $A$. 

In this paper we introduce a new family of posets, defined from a digraph together with a valuation on its vertices. Here, we focus exclusively on the case where the digraph is simple and acyclic, in which case the corresponding poset has a rich combinatorial structure. It appears that many well-known posets can be described within this theory, and after a careful case by case study, we show that the weak order on Coxeter groups of type $A$, $\widetilde{A}$ and $B$, the flag weak order on the wreath product $\mathbb{Z} _r \wr S_n$ (see \cite{ABR}), and the up-set (resp. down-set) lattice of any finite poset, admit such a description. 

The study of this family of posets will be further developed in \cite{FV2}, in which we will show how they can be used to study two long-standing conjectures of Matthew Dyer on the geometry of root systems in infinite Coxeter groups (see \cite{Dy93} and \cite{Dy11}). Moreover, in another subsequent publication we will highlight connections which exist between our theory and Tamari and Cabrian lattices.
Note that except for the case of Coxeter groups of type $A$, the content of \cite{FV2} will not overlap the content of the current paper. Indeed, \cite{FV2} is mainly centred on algebraic and geometric aspects of this construction, while here we develop the combinatorial ones: we give a new formula for the values of the M\"obius function and we provide a new combinatorial model for the maximal chains in the weak order on Coxeter groups of type $A$, $B$ and $\widetilde{A}$.

Our construction relies on a generalization of the notion of linear extension of a finite poset to simple acyclic digraph, and leads us to associate quasi-symmetric functions with each element of our posets, as it is the case for linear extensions in the context of $P$-partitions (see \cite{S3}). It seems that most of the functions associated with an element do not give much insights about the underlying poset structure. However, the form of the underlying digraph sometimes leads to a canonical choice among these quasi-symmetric functions, which occurs when considering the digraphs associated with types $A$ and $\widetilde{A}$. In these two cases, we show (following a similar methods as in \cite{FGR} and \cite{YY}) that the canonical quasi-symmetric functions which arise are exactly the Stanley and Lam's symmetric functions.

In the author's opinion, the connections between our construction and quasi-symmetric functions presented in Section~\ref{SecQuasi} just scratch the surface and would require a more exhaustive study. Moreover, our results suggest that the construction presented here could be generalized so as to obtain the weak order on any Coxeter group, and thus may lead to a generalization of Stanley symmetric functions to a wider class of Coxeter groups. A good starting point would be to look for a combinatorial description of type $B$ Stanley symmetric functions, using the digraph introduced here. \\

The paper is organised as follows: in Section~\ref{SecDef} we define the family of posets from \emph{valued digraphs}, which are couples of a simple acyclic digraph together with a valuation on its vertices. We exhibit some general properties of these posets, namely they are \emph{graded complete meet semi-lattices} in the general case, \emph{graded complete lattices} when the underlying digraph is finite, and we give a simple and explicit formula to compute the values of their \emph{M\"obius function}. In Section~\ref{SecWeak}, we show that the (right) weak order on Coxeter groups $A_{n-1}$, $B_n$ and $\widetilde{A}_n$, the flag weak order on $\mathbb{Z} _r \wr S_n$, and the up-set (resp. down-set) lattice of any finite poset can be described thanks to this theory.
In Section~\ref{SecQuasi}, we exhibit a link between these posets and the theory of quasi-symmetric functions. More precisely, we explain how the series associated with any \emph{$P$-partition} (see \cite{BM}), the Stanley symmetric functions, and Lam's generalization naturally arise from this description. \\

\noindent {\bf Acknowledgements.} Parts of the current paper have been accepted to FPSAC~2015 as an extended abstract (see \cite{FV3}).

\section{Definition of a new family of posets}\label{SecDef}
We begin with some standard definitions about poset theory and graph theory. A poset is a couple $\Pp=(P,\leq)$, where $P$ is a set and $\leq$ is a binary relation which is \emph{reflexive}, \emph{antisymmetric}, and \emph{transitive}. A poset $\Pp$ is called a \emph{complete meet (resp. join) semi-lattice} if and only if every subset $S$ of $\Pp$ has an \emph{infimum} (resp. \emph{supremum}) in $\Pp$, \emph{i.e.} there exists $z$ in $\Pp$ such that if $y \in \Pp$ and $y \leq x$ (resp. $x \leq y$) for all $x\in S$, then $y \leq z$ (resp. $z\leq y$). If $\Pp$ is both a complete join and meet semi-lattice, then we say that $\Pp$ is a \emph{complete lattice}. A \emph{lower set} of $\Pp$ is a subset $A$ of $\Pp$ such that for all $x \in A$ and $y$ in $P$, if $y \leq x$ then $y \in A$. It is classical that the lower sets of $\Pp$ ordered by inclusion form a lattice (see $\mathsection$3 in \cite{S2}).

A \emph{simple digraph} is a couple $G=(V,E)$, where $V$ is the set of \emph{vertices} of $G$, and $E$ is a subset of $V \times V$ called the set of \emph{arcs} of $G$. A \emph{cycle} of $G$ is a finite sequence $x_1,\ldots,x_n$ of vertices of $G$ such that $(x_i,x_{i+1})\in E$ for all $i \in \mathbb{N}$, where the indices are taken modulo $n$. A graph $G$ is called \emph{acyclic} if it does not have any cycle. On each vertex $x$ of $G$, we define the statistic $$d^+(x)=|\{ y \ | \ (x,y)\in E \}|,$$ called the \emph{out-degree} of $x$, and which is possibly infinite. \\

We now present a method to obtain all lower sets of a finite poset $\Pp$.
 Since $\Pp$ is finite, there exists $a_1$ in $\Pp$ which is a minimum, that is if $x \leq a_1$ in $\Pp$, then $x=a_1$. Let $\Pp_2=(P\setminus \{ a_1\},\leq)$ be the finite poset obtained by removing $a_1$ from $\Pp$. Then, there exists $a_2$ in $\Pp_2$ which is a minimum and we can define the poset $\Pp_3$ obtained by removing $a_2$ in $\Pp_2$, and so on. Finally, we end with an injective sequence $[a_1,\ldots,a_n]$ of elements of $\Pp$. This sequence is a linear extension of $\Pp$ by construction. Furthermore, one can easily prove by induction that all the linear extensions of $\Pp$ can be obtained by this way. 

In a certain sense, this method ``peels'' a finite poset element by element, in order to obtain a family of sequences which give rise to an interesting family of sets (here, the lower sets). Here, we propose to apply a similar principle to a simple acyclic digraph. Namely, we will peel the digraph vertex by vertex, with respect to a constraint given by a valuation on its vertices. It will give rise to a family of sequences of vertices of the graph, then to a family of subsets of vertices having an interesting poset structure once ordered by inclusion. 

We start with the definition of the valuation on the vertices of a simple acyclic digraph.

\begin{definition}
Let $G=(V,E)$ be a simple acyclic digraph. A valuation $\theta:V \rightarrow \mathbb{N}$ is called an \emph{out-degree compatible valuation} on $G$ (OCV) if and only if for all $x \in V$, we have 
\[0 \leq \theta(x) \leq d^+(x).\]
A pair $\G=(G,\theta)$, where $G$ is a simple acyclic digraph and $\theta$ is an OCV, is called a \emph{valued digraph}.
\end{definition}

In what follows, $\mathcal{G}=(G,\theta)$ will denote a valued digraph. Recall that our aim is to generalize the method which peels finite posets to valued digraphs. Thus, we first need to specify which vertices of a valued digraph can be peeled. This is the point of the following definition.

\begin{definition}[Erasable vertex]
A vertex $x$ of $G$ is called \emph{erasable} in $\mathcal{G}$ if and only if:
\begin{itemize}
\item $\theta(x)=0$;
\item for all $z \in V$ such that $(z,x) \in E$, we have $\theta(z) \neq 0$.
\end{itemize}
\end{definition}

We now introduce the \emph{peeling process}, which is indeed a generalisation of the process on finite posets presented in the introduction of this section.

\begin{definition}[Peeling process and peeling sequences]\label{DefPeelSeq}
Given $\G=(G,\theta)$ a valued digraph, we construct recursively two sequences: a sequence $L=[x_1,x_2,\ldots]$ of elements of $V$, and a sequence $(\G_i=(G_i,\theta_i))_{1\leq i}$ of valued digraphs as follows. 
\begin{enumerate}
\item Let $\G_1=\G$.
\item If there is not any erasable vertex in $\G_i$, the process stop. Otherwise, choose $x$ a vertex of $G_i$ which is erasable in $\G_i$, and set $x_i=x$.
\begin{enumerate}
\item Let $G_{i+1}$ be the simple acyclic directed graph obtained by removing the vertex $x_i$ in $G_i$ and all the arcs of the form $(z,x_i)$ or $(x_i,z)$ in $G_i$.
\item Let $\theta_{i+1}$ be the OCV on $G_{i+1}$ such that $\theta_{i+1}(y)=\theta_{i}(y)-1$ if $(y,x)$ is an arc of $G_i$, and $\theta_{i+1}(y)=\theta_i(y)$ otherwise. Then set $\G_{i+1}=(G_{i+1},\theta_{i+1})$ and iterate Step 2.
\end{enumerate}
\end{enumerate}
A sequence $L$ coming from this process is called a \emph{peeling sequence} of $\G$, and we denote by $PS(\G)$ the set of all peeling sequences of $\G$.
\end{definition}

Recall that the lower sets of any finite poset $\Pp$ are the initial sections of some linear extension $\Pp$, and we can extend this notion to valued digraph in the natural way.

\begin{definition}
Let $L=[x_1,x_2,\ldots]$ be a peeling sequence of $\G$. The \emph{initial sections} of $L$ are the sets of the form $\{x_1,x_2\ldots,x_k \}$, $k\in \mathbb{N}^*$. By convention, $\emptyset$ is an initial section of $L$. The set of the initial sections of all the peeling sequences of $\G$ will be denoted by $IS(\G)$.
\end{definition}

Finally, recall that the lower sets of any finite poset gives rise, once ordered by inclusion, to a classical lattice called its down-set lattice.
Once again, this concept naturally generalizes to valued digraphs, and the posets we will consider all along this paper are the posets $(IS(\G),\subseteq)$ for some valued digraphs $\G$.

\begin{example}
Consider $\G$ as depicted in the upper left corner of Figure~\ref{FigFPSACex}. The peeling sequences of $\G$ are $L_1=[a,c,b]$ and $L_2=[b,c,a]$, thus $$IS(\G)=\{ \emptyset, \{a \},\{b \},\{a,c \},\{b,c \},\{a,b,c \} \}.$$
\begin{figure}[!h]
\includegraphics[scale=1]{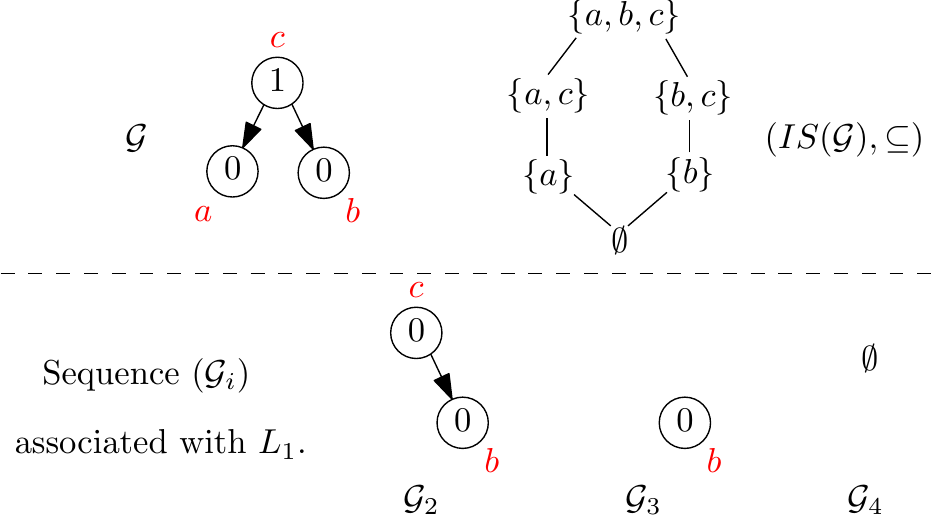}
\caption{}\label{FigFPSACex}
\end{figure}
\end{example}

We finish this section with stating our main results concerning the properties of $(IS(\G),\subseteq)$. The proofs are given in Section~\ref{SectionProof}.

\begin{Theorem}\label{Lattice}
Let $\G=(G,\theta)$ be a valued digraph, the poset $(IS(\G),\subseteq)$ is a graded complete meet semi-lattice, and its rank function is $\rho : A\rightarrow |A|$. Moreover, if $G$ is finite, then $(IS(\G),\subseteq)$ is a complete lattice.
\end{Theorem}

We also have an explicit formula for the values of the M\"obius function of the poset $(IS(\G),\subseteq)$.
For the sake of clarity, we give the formula only for the couples of the form $(\emptyset,A)$, $A \in IS(\G)$, but a similar one can be stated for all couples in $IS(\G)$.

\begin{Theorem}\label{mobius}
Let $A \in IS(\G)$, $\mathcal{N}(A)=\{x \in A \ | \ \theta(x)=0 \}$ and $\mathcal{F}(A)=\{x \in A \ | \ A\setminus \{x\} \in IS(\G) \}$, we have two cases.
\begin{enumerate}
\item If $\mathcal{F}(A)=\mathcal{N}(A)$, then $\mu(\emptyset,A)=(-1)^{|\mathcal{N}(A)|}$.
\item Otherwise, $\mu(\emptyset,A)=0$.
\end{enumerate}
\end{Theorem} 

\section{Proofs}\label{SectionProof}

In this section, we provide the proofs of Theorems~\ref{Lattice} and \ref{mobius}. Both proofs rely on an intrinsic characterization of the elements of $IS(\G)$ (Proposition~\ref{Technic}) and on a technical lemma (Lemma~\ref{TechLem}), which we give here.

\begin{proposition}\label{Technic}
Let $\G=(G,\theta)$ be a valued digraph and $A$ be a finite subset of vertices of $G$, then $A \in IS(\G)$ if and only if:
\begin{enumerate}
\item[(1)] for all $x\in A$, $\theta(x) \leq |\{ \ y \ | \ y\in A \ {\rm and} \ (x,y) \in E \}|$;
\item[(2)] for all $x \in V \setminus A$, $\theta(x) \geq |\{ \ y \ | \ y\in A \ {\rm and} \ (x,y) \in E \}|$.
\end{enumerate}
\end{proposition}

\begin{proof}
Assume that $A \in IS(\G)$, \emph{i.e.} there exists $L=[x_1,x_2,\ldots]\in PS(\G)$ and $k \in \mathbb{N}^*$ such that $\{x_1,\ldots,x_k \}=A$. Let $x$ be a vertex of $G$, then we divide our study into two cases. 
\begin{itemize}
\item If $x \in A$, then there exists $i \leq k$ such that $x_i=x$, and we obviously have $$\{ \ x_j \ | \ j<i \ {\rm and} \ (x_i,x_j) \in E \} \subseteq \{ \ y \ | \ y\in A \ {\rm and} \ (x_i,y) \in E \}.$$ Furthermore, by definition of the peeling process, we have $\theta(x_i)=|\{ \ x_j \ | \ j<i \ {\rm and} \ (x_i,x_j) \in E \}|$. Hence, $x$ satisfies point (1).
\item If $x \notin A$, then set $\G_i=(G_i,\theta_i)$ the sequence associated with $L$. By definition of the peeling process, we have $\theta_{k+1}(x)=\theta(x)-|\{ \ y \ | \ y\in A \ {\rm and} \ (x,y) \in E \}|\geq 0$. Hence, $x$ satisfies Point (2).
\end{itemize}

Conversely, assume that $A$ satisfies both points (1) and (2). We will prove that $A \in IS(\G)$ recursively on $k:=|A|$. If $k=1$, then set $x$ the vertex of $G$ such that $A=\{ x\}$. We have $\theta(x)=0$ thanks to Point (1), and Point (2) implies that for all $y \neq x$ such that $(y,x) \in E$ we have $\theta(y) \geq 1>0$. Hence, $x$ is erasable in $\G$, so that $A \in IS(\G)$. Let $k$ be such that the property is true and assume that $|A|=k+1$. We first prove that there exists a vertex in $A$ which is erasable. Since $G$ is acyclic and $A$ is finite, there exists $z \in A$ such that for all $y \in A$, $(z,y) \notin E$. Then, by Point (1), we have $\theta(z)=0$. Again, since $G$ is acyclic and $A$ is finite, there exists $x$ in $A$ such that $\theta(x)=0$ and for all $y \in A$, if $(y,x)\in E$ then $\theta(y)>0$. Furthermore, for all $z \in A$ such that $\theta(z)=0$, if there exists $y \in V \setminus A$ such that $(y,z) \in E$, then $\theta(y) > 0$ by Point (2). Consequently, this vertex $x$ is erasable and can be peeled at the first step of the peeling process. Thus, if we set $\G'$ the valued digraph obtained with the peeling process after we peeled the vertex $x$, then $|A\setminus \{x\}|=n$ and $A \setminus \{ x \}$ clearly satisfies Points (1) and (2) in $\G'$. Therefore, by induction we have $A \in IS(\G)$. 
\end{proof}

We finish with a technical lemma.

\begin{lemma}\label{TechLem}
Let $S \subseteq IS(\G)$ and denote by $X$ the set $\displaystyle{\bigcap_{A \in S}A}$. If there exists $x \in X$ such that $\theta(x)=0$, then there exists $z \in X$ which is erasable in $\G$.
\end{lemma}

\begin{proof}
Let $z \in X$ be such that $\theta(z) =0$. For all $y \notin X$, there exists $B \in S$ such that $y \notin B$. Therefore, if $(y,z)\in E$, then by Proposition~\ref{Technic} we have $\theta(y) \geq 1$. Assume by contradiction that for all $x \in X$ such that $\theta(x)=0$, there exists $y \in X$ such that $\theta(y)=0$ and $(y,x) \in E$. Since $X$ is finite, this implies that there is a cycle in $G$, which is absurd. Hence, there exists an erasable vertex of $\G$ in $X$, and this ends the proof. 
\end{proof}

\subsection{Proof of Theorem~\ref{Lattice}}

We divide the proof of Theorem~\ref{Lattice} into two distinct steps. First, we prove that $(IS(\G),\subseteq)$ is a graded poset (Proposition~\ref{PropGradedPoset}). Then, we prove that it is a meet semi-lattice (Corollary~\ref{CoroMeet}), constructing explicitly the infimum of any subset of $IS(\G)$. 

Let us begin with a proposition, which immediately implies that $(IS(\G),\subseteq)$ is graded.

\begin{proposition}\label{PropGradedPoset}
Let $A$ and $B$ be two elements of $IS(\G)$, and denote by $k$ and $q$ the cardinality of $A$ and $B$, respectively. If $A \subseteq B$, then there exists $L=[x_1,x_2,\ldots]\in PS(\G)$ such that $A=\{x_1,\ldots,x_k \}$ and $B=\{x_1,\ldots,x_q \}$. Consequently, $(IS(\G), \subseteq)$ is graded with rank function $A \mapsto |A|$.
\end{proposition}

\begin{proof}
Since the case $k=q$ is obvious, we assume that $k<q$. Let us now perform the peeling process on $\G$, in order to construct the claimed peeling sequence. Since $A \in IS(\G)$, we begin with constructing the sequence $L$ by peeling the first $k$ elements $x_1,\ldots,x_k$ in $A$. Consequently, $A=\{ x_1, \ldots , x_k \}$. 

Let us now consider the valued digraph $\G_{k+1}=(G_{k+1},\theta_{k+1})$ coming from the peeling process after we peeled $x_1,\ldots,x_k$. By definition of $\G_{k+1}$, $C=B \setminus A $ is a subset of vertices of $G_{k+1}$. We will prove that $C$ is in $IS(\G_{k+1})$ checking that $C$ satisfies both Points (1) and (2) of Proposition~\ref{Technic}. By construction, for all $z \in C$ we have $\theta_{k+1}(z)=\theta(z)-|\{ y \in A \ | \ (z,y) \in E \}|$, and 
$$\{ y \in C \ | \ (z,y) \in E \}= \{  y \in B \ | \ (z,y) \in E \} \setminus \{  y \in A \ | \ (z,y) \in E \}.$$
Since $B \in IS(\G)$, we have by Proposition~\ref{Technic} that $\theta(z) \leq |\{  y \in B \ | \ (z,y) \in E \}|$, so that $$\theta_{k+1}(z) \leq |\{  y \in B \ | \ (z,y) \in E \}|-|\{ y \in A \ | \ (z,y) \in E \}|=|\{ y \in C \ | \ (z,y) \in E \}|.$$ Then, $C$ satisfies Point (1) of Proposition~\ref{Technic}. Using similar arguments, we show that $C$ also satisfies Point (2) of Proposition~\ref{Technic}. Hence, there exists a peeling sequence $[x_{k+1},x_{k+2},\ldots]$ of $\G_{k+1}$ such that $C=\{ x_{k+1}, \ldots , x_q\}$, and finally, the sequence $L=[x_1,x_2,\ldots]$ is a peeling sequence of $\G$ such that $A=\{x_1,\ldots , x_k \}$ and $B=\{x_1,\ldots , x_q \}$. This ends the proof of the proposition.
\end{proof}

We end the proof of Theorem~\ref{Lattice} showing that $(IS(\G),\subseteq)$ is a meet semi-lattice. For that purpose,  we construct explicitly the infimum (also called the \emph{meet}) of a set $S \subseteq IS(\G)$. \\

\noindent {\bf Construction of the meet.}
Let $S$ be a subset of $IS(\G)$ and $X$ be the intersection of all the elements of $S$, we will construct recursively a set $C \in IS(\G)$ as follows. 

If for all $x \in X$, $\theta(x) \neq 0$, we set $C=\emptyset$. Otherwise, let $z_1 \in X$ be an erasable vertex of $\G$ and start the peeling process by peeling this vertex. We denote by $\G_2=(G_2,\theta_2)$ the obtained valued digraph. Then, for all $A \in S$, we have $A\setminus \{ z_1\} \in IS(\G_2)$. Therefore, we can again apply Lemma~\ref{TechLem} to $X\setminus \{z_1\}$ seen as a subset of vertices of $G_2$: if for all $x \in X\setminus \{z_1\}$ we have $\theta_2(x)\neq 0$, then we set $C=\{ z_1 \}$; otherwise, let $z_2\in X\setminus \{z_1\}$ be an erasable vertex of $\G_2$ and perform the peeling process peeling this $z_2$ in $\G_2$. We repeat this procedure until there is not any erasable vertex left (this process always ends, since $X$ is finite), and we set $C$ the resulting set. By construction, $C \in IS(\G)$. \\

At first glance, this set $C$ does not appear to be well defined, and seems to depend heavily on the choices of vertices made at each step of its construction. The next proposition shows that is not the case.

\begin{proposition}\label{PropVerifMeet}
Let $A \in IS(\G)$. If $A \subseteq X$, then $A \subseteq C$.
\end{proposition}

\begin{proof}
We split the proof into two cases.
\begin{itemize}
\item If for all $z \in X$ we have $\theta(z) \neq 0$, then, by definition of the peeling process, we have that $A = \emptyset=C$. 
\item If there exists $z \in X$ such that $\theta(z)=0$, then $C \neq \emptyset$. Thus there exists $L=[x_1,x_2,\ldots] \in PS(\G)$ and $k \in \mathbb{N}^*$ such that $C=\{x_1,\ldots,c_k \}$. Let us denote by $\G_i=(G_i,\theta_i)_i$ the sequence of valued digraphs associated with $L$ and let $L'=[z_1,z_2,\ldots]$ be in $PS(\G)$ such that $A=\{z_1,\ldots,z_{|A|}\}$. 

Assume by contradiction that $A \not\subset C$ and consider $j\leq|A|$ minimal such that $z_j \notin C$. We have that $z_j$ is a vertex of $G_{k+1}$ and, by minimality, for all $q<j$ there exists $1 \leq i_q \leq k$ such that $z_q=x_{i_q}$. 
Let us now compute the value of $\theta_{k+1}(z_j)$. By definition of the peeling process, we have
$$\theta(z_j)=|\{q<j \ | \ (z_j,z_q) \in E \}|=|\{q<j \ | \ (z_j,x_{i_q})\in E \}|.$$
However, for all $q<j$ we have $i_q \leq k$, so that
$$\theta_{k+1}(z_j)=\theta(z_j)-|\{ p \leq k \ | \ (z_j,x_p)\in E \}| \leq 0.$$
Thus, $\theta_{k+1}(z_j)=0$ and this is absurd by construction of $C$, hence $A \subseteq C$.
\end{itemize}
This ends the proof
\end{proof}

As $C \subseteq X$ by construction, Proposition~\ref{PropVerifMeet} implies the following corollary.

\begin{corollary}\label{CoroMeet}
The set $C$ is the infimum of $S$.
\end{corollary}

Finally, note that if the underlying graph $G$ is finite, then we obviously have that $V$, the set of all the vertices of $G$, satisfies Points (1) and (2) of Proposition~\ref{Technic}. Consequently, $V \in IS(\G)$ and the poset $(IS(\G),\subseteq)$ is bounded. Thus, $(IS(\G),\subseteq)$ is a lattice since it is a meet semi-lattice (more precisely, it is a complete lattice), and this ends the proof of Theorem~\ref{Lattice}.

\subsection{Proof of Theorem~\ref{mobius}}

The proof of this formula is purely combinatorial, and is based on the well-known \emph{Inclusion-Exclusion Principle} (see $\mathsection$2 in \cite{S2}). We first introduce some notations: for all $A \in IS(\G)$, we denote by $[\emptyset,A]$ the set $\{B \in IS(\G) \ | \ \emptyset \subseteq B \subseteq A \}$.
Let $A$, $\mathcal{N}(A)$ and $\mathcal{F}(A)$ be as defined in Theorem~\ref{mobius}, for all $S \subseteq \mathcal{F}(A)$, let us denote by $A_S$ the infimum of $\{A\setminus \{x\} \ | \ x\in S \}\subseteq IS(\G)$.

We begin the proof with a technical lemma.

\begin{lemma}\label{LemMob}
Let $S \subseteq \mathcal{F}(A)$, we have $A_S \neq \emptyset$ if and only if $\mathcal{N}(A) \not\subset S$.
\end{lemma}

\begin{proof}
Obviously, $\displaystyle{\bigcap_{x \in S} (A\setminus\{x\})= A \setminus S}$. Thanks to Lemma~\ref{TechLem}, if $\mathcal{N}(A) \not\subset S$, then there exists $z \in A\setminus S$ which is erasable in $\G$, so that the infimum of $\{A\setminus \{x\} \ | \ x\in S \}$ is not $\emptyset$. The proof of the converse implication is based on similar arguments.
\end{proof}

An immediate consequence of the meet semi-lattice structure of $\Pp(\G)$ is that, for all $A$ and $B$ in $IS(\G)$, $[\emptyset,A]\cap[\emptyset,B]=[\emptyset,A \wedge B]$ where $A\wedge B$ is the infimum of $\{A,B\}$. This basic remark leads to the claimed formula.

First, we have $$\displaystyle[\emptyset,A]\setminus\{A\} =\bigcup_{x \in \mathcal{F}(A)}[\emptyset,A_{\{x\}}].$$ Then, by the Inclusion-Exclusion Principle, we have $$\left| [\emptyset,A]\setminus\{A\}\right| =\sum_{\emptyset \neq S\subseteq \mathcal{F}(A)}(-1)^{|S|+1}\ \left| \bigcap_{x \in S}[\emptyset,A_{\{x\}}] \right|=\sum_{\emptyset \neq S\subseteq \mathcal{F}(A)}(-1)^{|S|+1}\ \left| [\emptyset,A_{S}] \right|.$$
Once applied to the M\"obius function of $\Pp(\G)$, this gives rise to the following identity: 
\begin{equation}\label{eqMob}
\mu(\emptyset,A)=-\sum_{\emptyset \neq S\subseteq \mathcal{F}(A)}(-1)^{|S|+1}\sum_{B\in[\emptyset,A_S]}\mu(\emptyset,B).
\end{equation}
By definition of the M\"obius function, $\displaystyle \sum_{B\in[\emptyset,A_S]}\mu(\emptyset,B)=1$ if $A_S=\emptyset$, and 0 otherwise. Hence, thanks to Lemma~\ref{LemMob}, if $\mathcal{N}(A)\not\subset\mathcal{F}(A)$, then $\mu(\emptyset,A)=0$. Otherwise, Equation~(\ref{eqMob}) becomes
\begin{align*}
\mu(\emptyset,A)&=-\sum_{\mathcal{N}(A) \subseteq S\subseteq \mathcal{F}(A)}(-1)^{|S|+1}=(-1)^{|\mathcal{N}(A)|}\sum_{S \subseteq \mathcal{F}(A)\setminus\mathcal{N}(A)}(-1)^{|S|} \\
 &=(-1)^{|\mathcal{N}(A)|}(1-1)^{|\mathcal{F}(A)\setminus\mathcal{N}(A)|}.
\end{align*}
Theorem~\ref{mobius} follows immediately.

\section{Link with the weak order}\label{SecWeak}

In this section, we show several examples of classical posets which can be described using valued digraphs. We first recall the definition of weak order on a Coxeter groups.

Let $W$ be a Coxeter group with generating set $S$, the \emph{weak order} on $W$ is the poset
$(W,\leq_R)$, defined as follows: we say that $w \leq_R \tau$ if and only if there exists $s_1,\ldots,s_k$ in $S$ such that $\tau=\omega s_1 \cdots s_k$ and $\ell(\tau)=\ell(w)+k$. 

It is well-known that $(W,\leq_R)$ is a complete meet semi-lattice when $W$ is infinite, a complete lattice when $W$ is finite, and that its M\"obius function takes values into $\{\pm 1,0 \}$ (see \cite{Bjor} and \cite{BB}). Hence, it is natural to look for an interpretation of the weak order through the theory developed in previous sections. Indeed, such an interpretation exists in some cases, and we give an explicit description for the following list of posets.

\begin{Theorem}\label{TheoListPoset}
For each poset $(\Pp,\leq)$ in the following list, there exists an explicit valued digraph $\G$ such that $(\Pp,\leq)$ is isomorphic to $(IS(\G),\subseteq)$:
\begin{itemize}
\item $(W,\leq_R)$ where  $W=A_{n-1},B_n,\widetilde{A}_n$ and $\leq_R$ is the (right) weak order on $W$;
\item $(\mathbb{Z} _r \wr S_n,\leq_f)$, called the flag weak order on $\mathbb{Z} _r \wr S_n$;
\item the up-set (resp. down-set) lattice of any finite poset.
\end{itemize}
\end{Theorem}

We prove Theorem~\ref{TheoListPoset} with a careful case-by-case study, which is done in the following sections. 
More precisely, in Section~\ref{SectSym} we provide a candidate of valued digraph associated with the weak order on $A_{n-1}$, and we prove that this candidate indeed provides a description of the weak order on $A_{n-1}$ in Section~\ref{SectionSymProof}. Similarly, in Section~\ref{SectB} we construct a valued digraph associated with $B_n$, and we prove in Section~\ref{SectionTypeBProof} that this valued digraph describe the weak order on $B_n$. Section~\ref{SectAffineA} is devoted to the study of the weak order on $\widetilde{A_n}$. Finally, in Sections~\ref{SectionFlagWeak} and \ref{SectionDownSet} we study the cases of the flag weak order and the up-set lattice, respectively.

\subsection{Weak order on $A_{n-1}$}\label{SectSym}

Recall that $A_{n-1}$ is the Coxeter group with generating set $S=\{s_1,\ldots,s_{n-1} \}$ and with Coxeter matrix $M = (m_{st})_{s,t\in S}$ given by $m_{s_is_{i+1}}=3$ for all $1 \leq i \leq n-2$, and $m_{st}=2$ otherwise. As usual, we identify $A_{n-1}$ with the symmetric group $S_n$, identifying the generator $s_i$ with the \emph{simple transposition} of $S_n$ which exchanges the integers $i$ and $i+1$.\\

When we try to find a valued digraph $\G=(G,\theta)$ such that $(A_{n-1},\leq_R)$ is isomorphic to $(IS(\G),\subseteq)$, the first problem arising is that, on the one hand we have a poset whose elements are permutations, and on the other hand we have a poset whose elements are sets. In order to overcome this difficulty, let us consider a canonical set associated with each permutation $\sigma \in S_n$, its \emph{inversion set}:
\begin{equation}\label{DefInv}
{\rm Inv}(\sigma)=\{(a,b) \in [n]^2\ | \ a<b \ {\rm and} \ \sigma^{-1}(a) > \sigma^{-1}(b) \}.
\end{equation}

There is a deep connection between inversion sets and the weak order on $S_n$. That is, we have the following well-known property (see \cite{BB}). For any $\sigma,\omega \in S_n$,
\begin{equation}\label{InvWeakSym}
 \sigma \leq_R \omega {\rm \ if \ and \ only \ if \ } {\rm Inv}(\sigma) \subseteq {\rm Inv}(\omega).
\end{equation}
This property allows us to clarify our goal: we are looking for a valued digraph $\G=(G,\theta)$ such that,
\begin{enumerate}
\item the vertices of the graph are indexed by couples of integers $(a,b)\in [n]^2$ such that $a<b$;
\item the digraph structure of $G$, together with the valuation $\theta$, imply that $IS(\G)$ is constituted exactly of the sets of the form ${\rm Inv}(\sigma)$, $\sigma \in S_n$. 
\end{enumerate}

There is a convenient way to represent the set $\{(a,b)\in [n]^2 \ | \ a<b \}$, considering the \emph{$n$-th staircase diagram}, namely the Ferrers diagram of the partition $\lambda_n=(n-1,n-2,\ldots,1)$ of size $N={n \choose 2}$. On the left of Figure~\ref{FigFPSAC03}, the diagram associated to the case $n=5$ is represented. The coordinates of each box can be read thanks to the circled integers on the diagonal. From now on, we identify $\lambda_n$ with the set $\{(a,b)\in [n]^2 \ | \ a<b \}$.
\begin{figure}[!h]
\includegraphics[scale=1]{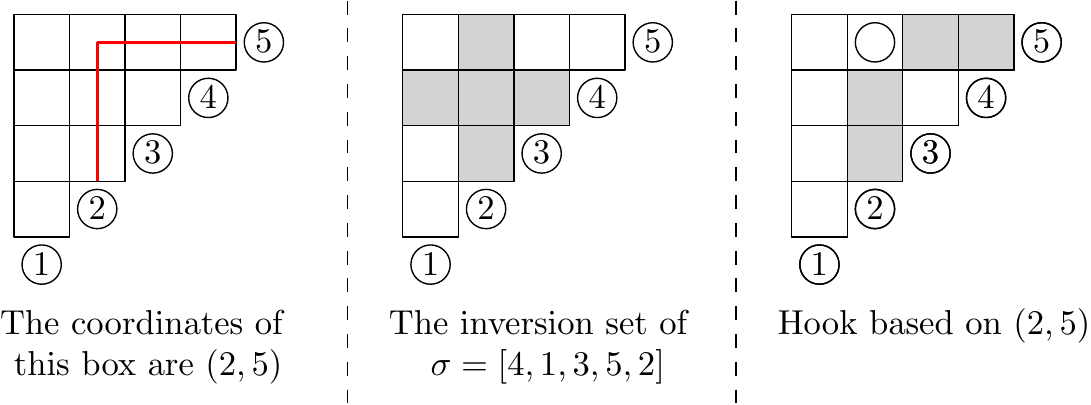}
\caption{}\label{FigFPSAC03}
\end{figure}
As shown in the middle of Figure~\ref{FigFPSAC03}, one can easily visualize the inversion set of any element of $S_n$ as a subset of boxes in $\lambda_n$. Note that the set made of all the boxes of the diagram corresponds to the inversion set of the \emph{reverse permutation} $[n,n-1,\ldots,1] \in S_n$, which is the maximal element in the weak order. 

We can define a digraph structure $G$ on the staircase diagram $\lambda_n$ (where the vertices are the boxes of the diagram), thanks to a classical combinatorial object associated to each box $\mathfrak{c}\in \lambda_n$, \emph{the hook based on $\mathfrak{c}$}, denoted $H(\mathfrak{c})$, consisting of $\mathfrak{c}$ and all the boxes which are on the right and below $\mathfrak{c}$ (see Figure~\ref{FigFPSAC03}, on the right): we say that there is an arc from $\mathfrak{c}$ to $\mathfrak{d}$ if and only if $\mathfrak{c} \neq \mathfrak{d}$ and $\mathfrak{d} \in H(\mathfrak{c})$. Obviously, the resulting digraph $G$ is acyclic, and the out-degree of any box is an even number. Thus, if we set $\theta$ the function defined by $\theta(\mathfrak{c}):=\frac{d^+(\mathfrak{c})}{2}$, then the couple $\A=(G,\theta)$ is a valued digraph. 
Let us summarize this construction in a definition.

\begin{definition}\label{DefValGrSym}
Let $G=(V,E)$ be the digraph such that $$V:=\lambda_n=\{(a,b)\in[n]^2 \ | \ a<b \} \ {\rm and} \ E:=\{(\mathfrak{c},\mathfrak{d})\in \lambda_n^2 \ | \ \mathfrak{c}\neq \mathfrak{d} \ {\rm and} \ \mathfrak{d} \in H(\mathfrak{c}) \}.$$
 We denote by $\A:=(G,\theta)$ the valued digraph such that for all $\mathfrak{c} \in \lambda_n$
 \[  \theta(\mathfrak{c}):=\frac{d^+(\mathfrak{c})}{2}.\]
\end{definition}

One can check that the posets $(IS(\A),\subseteq)$ obtained in the cases $n=2$, 3 and 4 are isomorphic to the weak order on $A_1$, $A_2$ and $A_3$, respectively. As stated in the following theorem, this situation is in fact general.

\begin{Theorem}\label{SymCase}
The posets $(S_n,\leq_R)$ and $(IS(\A),\subseteq)$ are isomorphic.
\end{Theorem}

The next section is dedicated to the proof of Theorem~\ref{SymCase}.

\subsection{Proof of Theorem~\ref{SymCase}}\label{SectionSymProof}

In this section, we will show that $$IS(\A)=\{ \Inv(\sigma) \ | \ \sigma \in S_n \},$$ which immediately implies Theorem~\ref{SymCase}. We divide our proof into three steps. First, for any permutation $\sigma \in S_n$ we define a statistic $d_{\sigma}$ on $\lambda_n\setminus \Inv(\sigma)$ (Definition~\ref{DefStatisticSym}), which we characterize using the notion of \emph{adjacency} (see Lemma~\ref{LemSymTech}). We then use this to give a combinatorial interpretation of the valuations $\theta_i$ appearing when we perform the peeling process on $\A$ (see Proposition~\ref{PropSymCombinatInt}). Finally, using this combinatorial interpretation, we prove that $IS(\A)=\{ \Inv(\sigma) \ | \ \sigma$ $ \in S_n \}$ (see Proposition~\ref{TechSym} and Corollary~\ref{CorolFinalA}). 

We begin with the definition of the statistic $d_{\sigma}$.

\begin{definition}\label{DefStatisticSym}
Let $\sigma \in S_n$ and $(a,b) \in \lambda_n \setminus \Inv(\sigma)$. Then, we set: 
$$d_{\sigma}(a,b):=|\{ a<k<b \ | \ \sigma^{-1}(a) < \sigma^{-1}(k)<\sigma^{-1}(b) \}|.$$
\end{definition}

Let $\sigma \in S_n$ and $(a,b) \in \lambda_n$, we say that $a$ and $b$ are \emph{adjacent} in $\sigma$ if and only if $\sigma^{-1}(b)=\sigma^{-1}(a)+1$. 
One can clearly visualize two adjacent entries of a permutation $\sigma\in S_n$, using the \emph{window notation} of $\sigma$. That is, $a$ and $b$ are adjacent in $\sigma$ if and only if $(a,b) \in \lambda_n$ and 
\[\sigma=[\sigma(1),\sigma(2),\ldots,a,b,\ldots,\sigma(n-1),\sigma(n)]. \]
This notion is linked to the weak order thanks to the following well-known property: for all $\sigma,\omega \in S_n$ we have $\sigma \leq_R \omega$ and $\ell(\omega)=\ell(\sigma)+1$ if and only if $\omega$ can be obtained from $\sigma$ by swapping positions of two adjacent entries of $\sigma$ and we say that $\omega$ \emph{covers} $\sigma$, denoted by $\sigma \lhd_R \omega$. This can be transposed to the context of inversion sets using Equation~\eqref{InvWeakSym} as follows: for all $\sigma,\omega \in S_n$, we have $$\sigma \lhd_R \omega {\rm \ if \ and \ only \ if \ } {\rm Inv}(\omega)={\rm Inv}(\sigma) \cup \{(a,b) \}$$ for $a$ and $b$ two adjacent entries of $\sigma$.

\begin{Remark}
Note that if $a$ and $b$ are two adjacent entries of $\sigma$, then $d_{\sigma}(a,b)=0$, but the converse is not true.
\end{Remark}

We now provide a characterization of the statistic $d_{\sigma}$.

\begin{lemma}\label{LemSymTech}
Let $\sigma \in S_n$ and $(a,b)\in \lambda_n \setminus \Inv(\sigma)$. Then, we have $$d_{\sigma}(a,b)=|\{ a<k<b \ | \ k \in \mathbb{N} \}|-|\{ a<k<b \ | \ (a,k) \in \Inv(\sigma) \ {\rm or} \ (k,b)\in \Inv(\sigma) \}|.$$
\end{lemma}

\begin{proof}
Let $s_1\cdots s_q$ be a reduced decomposition of $\sigma$ and denote by $\sigma_i$ the permutation $s_1\cdots s_i$, $0 \leq i \leq q$ (with the convention that $\sigma_0=Id$). We will prove by induction on $i$ that the lemma is true for $\sigma_i$.

Note that the property is obviously true for $\sigma_0$. Let $i\geq 0$ be such that the property is true.
For the sake of clarity, let us denote by $\delta_j$ the integer $d_{\sigma_j}(a,b)$. Since $s_1 \cdots s_q$ is reduced, we have $\sigma_i \lhd_R \sigma_{i+1}$, thus there exists a unique $(a_{i+1},b_{i+1})$ in $\Inv(\sigma_{i+1}) \setminus \Inv(\sigma_{i})$. We now show how one can deduce the value of $\delta_{i+1}$ from both $\delta_i$ and $(a_{i+1},b_{i+1})$.
\begin{itemize}
\item (case $(a_{i+1},b_{i+1})=(a,k)$ with $a<k<b$) the permutation $\sigma_{i+1}$ is obtained from $\sigma_i$ by exchanging the position of the integer $a$ with the position of the integer $k$. Moreover, since $\sigma_{i}\lhd_R \sigma_{i+1}$, $a$ and $k$ are adjacent in $\sigma_{i}$. However, $(a,b) \notin \Inv(\sigma_i)$, thus $k$ lies strictly between $a$ and $b$ in the window notation of $\sigma_i$, \emph{i.e.} we have 
\[\sigma_i=[\sigma(1),\ldots,a,k,\ldots,b,\ldots,\sigma(n)].\] 
Hence, it is no longer the case in $\sigma_{i+1}$, so that $\delta_{i+1}=\delta_i-1$.
\item If $(a_{i+1},b_{i+1})=(k,b)$ with $a<k<b$, then with similar arguments we show that $\delta_{i+1}=\delta_i-1$.
\item In all other cases, both $a_{i+1}$ and $b_{i+1}$ either lie between $a$ and $b$ in $\sigma_i$, or they do note, and this is also true for $\sigma_{i+1}$. Therefore, we have $\delta_{i+1}=\delta_i$.
\end{itemize}
Finally, by induction hypothesis, $\sigma_{i+1}$ satisfies the property, and this ends the proof.
\end{proof}

For the sake of clarity, we introduce the following useful notation.

\begin{definition}\label{DefUltraImp}
Let $\G=(G,\theta)$ be a valued digraph and $A \in IS(\G)$, we denote by $\G_A=(G_A,\theta_A)$ the valued digraph obtained after removing all the elements of $A$ in $\G$ using the peeling process.
\end{definition}

We are now able to provide a combinatorial interpretation of $\theta_A$ for some $A \in IS(\A)$.

\begin{proposition}\label{PropSymCombinatInt}
Let $A \in IS(\A)$, if there exists $\sigma \in S_n$ such that $A=\Inv(\sigma)$, then for all $(a,b) \in \lambda_n \setminus \Inv(\sigma)$, we have $\theta_A(a,b)=d_{\sigma}(a,b)$.
\end{proposition}

\begin{proof}
Let $(a,b) \in \lambda_n \setminus \Inv(\sigma)$, by construction of $\A=(G,\theta)$, there is an arc from $(a,b)$ to $(c,d)$ if and only if $(c,d)=(a,k)$ or $(k,b)$ with $a<k<b$. Thus, by definition of the peeling process, we have 
$$\theta_A(a,b)=\theta(a,b)-|\{ a<k<b \ | \ (a,k) \in \Inv(\sigma) \ {\rm or} \ (k,b)\in \Inv(\sigma) \}|.$$ 
Moreover, we obviously have $\theta(a,b)=b-a-1=|\{a<k<b \ | \ k \in \mathbb{N} \}|$. Consequently, thanks to Lemma~\ref{LemSymTech}, we have 
\begin{align*}
\theta_A(a,b)&=|\{a<k<b \ | \ k \in \mathbb{N} \}|-|\{ a<k<b \ | \ (a,k) \in \Inv(\sigma) \ {\rm or} \ (k,b)\in \Inv(\sigma) \}| \\
&=d_{\sigma}(a,b),
\end{align*}
which ends the proof.
\end{proof}

Finally, we are now able to prove the main property of this section, which immediately leads to the proof of Theorem~\ref{SymCase} (see Corollary~\ref{CorolFinalA}).

\begin{proposition}\label{TechSym}
Let $A \in IS(\A)$, $\sigma \in S_n$ such that $A=\Inv(\sigma)$ and $(a,b)\in \lambda_n\setminus A$. Then, $a$ and $b$ are adjacent in $\sigma$ if and only if $(a,b)$ is erasable in $\A_A$.
\end{proposition}

\begin{proof}
Assume that $a$ and $b$ are adjacent in $\sigma$, then $d_{\sigma}(a,b)=0$. Let $(c,d) \in \lambda_n \setminus A$ be such that there is an arc from $(c,d)$ to $(a,b)$, thus we have $(c,d)=(a,p)$ with $p>b$ or $(c,d)=(q,b)$ with $q<a$. Since $a$ and $b$ are adjacent in $\sigma$, we have in the first case that $b$ is between $a$ and $p$ in the window notation of $\sigma$, \emph{i.e} we have 
\[ \sigma=[\sigma(1),\ldots,a,b,\ldots,p,\ldots,\sigma(n)],\]
and we have in the second case that $a$ is between $q$ and $b$ in $\sigma$. In both cases, $\theta_{A}(c,d)=d_{\sigma}(c,d) \geq 1$. Consequently, $(a,b)$ is erasable in $\A_A$.

We now prove the converse implication. Let $(a,b) \in \lambda_n \setminus A$ be erasable in $\A_A$, and assume by contradiction that $a$ and $b$ are not adjacent in $\sigma$. Then, there exists $1 \leq c \leq n$ which is between $a$ and $b$ in $\sigma$ and since $\theta_A(a,b)=d_{\sigma}(a,b)=0$, we have $c<a$ or $c>b$. 
\begin{itemize}
\item Case $c<a$. Let $d$ be maximal such that $d<a$ and $d$ is between $a$ and $b$ in $\sigma$, and let $k$ be an integer which is between $d$ and $b$ in $\sigma$ (if such a $k$ exists), we have:
\begin{itemize}
\item by maximality of $d$, $k \notin \{d,d+1,\ldots,a\}$;
\item since $d_{\sigma}(a,b)=0$, $k \notin \{a,a+1,\ldots,b\}$.
\end{itemize}
Thus, $d_{\sigma}(d,b)=0=\theta(d,b)$, which is absurd since $(a,b)$ is erasable and there is an arc from $(d,b)$ to $(a,b)$.
\item The case $c>b$ leads to a similar contradiction.
\end{itemize}
This proves that $a$ and $b$ are adjacent in $\sigma$, and this ends the proof.
\end{proof}

\begin{corollary}\label{CorolFinalA}
$IS(\A)=\{ \Inv(\sigma) \ | \ \sigma \in S_n \}$.
\end{corollary}

\begin{proof}
Let $L=[(a_1,b_1),\ldots,(a_N,b_N)] \in PS(\A)$, since $\Inv(Id)=\emptyset$, $a_1$ and $b_1$ are adjacent in $Id$ by Proposition~\ref{TechSym}. Let $\sigma_1$ be the permutation which has $\{(a_1,b_1) \}$ as inversion set, then, using recursively Proposition~\ref{TechSym}, we show that for all $1 \leq k \leq N$, there exists a permutation $\sigma_k$ which has $\{(a_1,b_1),\ldots,(a_k,b_k) \}$ as inversion set. This is enough to prove the corollary.
\end{proof}

This concludes the proof of Theorem~\ref{SymCase}.

\subsection{Weak order on $B_n$}\label{SectB}

Recall that $B_n$ is the Coxeter group with generating set $S=\{s_0,s_1,\ldots, s_{n-1} \}$, and with Coxeter matrix $M = (m_{st})_{s,t\in S}$ given by $m_{s_is_{i+1}}=3$ for all $1 \leq i \leq n-1$, $m_{s_0s_1}=4$ and $m_{st}=2$ otherwise. This group can be seen as the group of the \emph{signed permutations} $\omega$ of the set $[\pm n]:=\{-n,\ldots,-1,1,\ldots,n \}$ satisfying $\omega(-m)=-\omega(m)$ for all $m$. Within this interpretation, $s_0$ is the signed permutation such that $s_0(1)=-1$ and $s_0(j)=j$ for all $j>1$, and $s_i$ is the permutation which exchange the positions of $i$ and $i+1$ (and also the positions of $-i$ and $-i-1$). In what follows, we will sometimes represent an element $\omega$ in $B_n$ by its \emph{full window notation}, that is:
\[ [\omega(-n),\omega(-(n-1)),\ldots,\omega(-1),\omega(1),\ldots,\omega(n-1),\omega(n)]. \]

Our aim in this section is to provide an interpretation of $(B_n,\leq_R)$ using our theory. First, we need to find a candidate of valued digraph. For that purpose, we follow the same method as in Section~\ref{SectSym}, using a good notion of inversion set.

\begin{Remark}
It is important to note that we will not use the notion of inversion set coming from root systems here. Indeed, the combinatorial techniques we use here heavily depend on the interpretation of $B_n$ as a set of permutations, and not as a set of reflections. The drawback of this approach is that we will have to relate by ourselves these inversion sets to the weak order on $B_n$. Fortunately, most of the technical points have already been accomplished in \cite{BB}.
\end{Remark}

We begin with associating to each element $\omega$ of $B_n$ a $B$-inversion set, defined by:
\begin{equation}\label{EqInvB} 
\Inv_B(\omega)=\{(a,b) \in [\pm n]^2 \ | \ a<b, \ |a| \leq b \ {\rm and} \ \omega^{-1}(a) > \omega^{-1}(b) \}.
\end{equation}

Let us now relate $B$-inversion set to the weak order on $B_n$. For that purpose, we will need a definition and a result coming from \cite{BB}, which we now give.

\begin{definition}see~\cite[Eq.(8.2)]{BB}
Let $\omega \in B_n$, the $B$-inversion number of $\omega$ is the quantity
\begin{align*}
\inv_B(\omega):=&|\{(a,b)\in[n]^2\ |\ a<b \ {\rm and} \ \omega(a)>\omega(b) \}| \\ & + \ |\{(a,b)\in[n]^2\ |\ a\leq b \ {\rm and} \ \omega(-a)>\omega(b) \}.|
\end{align*}
\end{definition}

\begin{lemma}\label{LemBrentBjornTypeB}
(see \cite[Eq. (8.6)]{BB})
Let $\omega \in B_n$ and $i\in [n-1]$, we have
\begin{equation*}
\inv_B(\omega s_i)=\left\{
\begin{aligned}
\inv_B(\omega)+1, \ & \ \text{if} \ \omega(i)<\omega(i+1), \\
\inv_B(\omega)-1, \ & \ \text{if} \ \omega(i)>\omega(i+1).
\end{aligned}
\right.
\end{equation*}
We also have $\inv_B(\omega s_0)=\inv_B(\omega)+\text{sign}\left(\omega(1)\right)$.
\end{lemma}

The statistic $\inv_B$ is related to $B$-inversion sets, thanks to the following lemma.

\begin{lemma}\label{LemInvB-invB}
For all $\omega \in B_n$, we have $\inv_B(\omega)=|\Inv_B(\omega)|$.
\end{lemma}

\begin{proof}
We have 
\begin{align*}
\inv_B(\omega)=& \ |\{(a,b)\in[n]^2\ |\ a<b \ {\rm and} \ \omega(a)>\omega(b) \}| \\ 
& \qquad + \ |\{(a,b)\in[n]^2\ |\ a\leq b \ {\rm and} \ \omega(-a)>\omega(b) \}.| \\ 
=& \ |\{(a,b)\in[\pm n]^2 \ | \ a<b, \ |a|\leq b \ \text{and} \ \omega(a)>\omega(b) \} \\
=& \ |\Inv_B(\omega)|,
\end{align*}
which is the expected result.
\end{proof}

We now begin to prove that $B$-inversion sets can be used to study $(B_n,\leq_R)$. That is, we will show that for all $\omega,\tau \in B_n$ we have
\begin{equation}\label{EqInvBWeak}
 \omega \leq_R \tau \ \text{if and only if} \ \Inv_B(\omega) \subseteq \Inv_B(\tau).
\end{equation}

We start with defining the equivalent of the notion of adjacency in type $B$.

\begin{definition}\label{DefAdjaB}
Let $\omega \in B_n$ and $a<b$ be in $[\pm n]$, we say that $a$ and $b$ are \emph{$B$-adjacent} in $\omega$ if and only if the following two conditions are true:
\begin{enumerate}
\item $|a| \leq b$,
\item $a$ and $b$ are adjacent in $\omega$ (seen as a permutation of $[\pm n]$, \emph{i.e.} the full window notation of $\omega$ is of the form
\[ \omega=[\omega(-n),\ldots,a,b,\ldots,\omega(n)]. \]
\end{enumerate}
\end{definition}

It appears that the notion of $B$-adjacency plays the same role in type $B$ as the usual notion of adjacency do in type $A$, as shown in the next proposition.

\begin{proposition}\label{propTechInvB}
Let $\omega \in B_n$ and $0 \leq j \leq n-1$, there exists $(a,b)$ such that $|a|\leq b$, and such that $\omega s_j$ is obtained from $\omega$ by swapping the positions of $a$ and $b$ and the positions of $-b$ and $-a$ in $\omega$. Then, we have two possibilities:
\begin{itemize}[leftmargin=0.5cm]
\item if $a$ and $b$ are $B$-adjacent in $\omega$, then $\ell(\omega s_j)=\ell(\omega)+1$ and $\Inv_B(\omega s_j)=\Inv_B(\omega)\cup \{(a,b) \}$;
\item if $a$ and $b$ are not $B$-adjacent in $\omega$, then $\ell(\omega s_j)= \ell(\omega)-1$ and $\Inv_B(\omega s_j)=\Inv_B(\omega)\setminus \{(a,b) \}$.
\end{itemize}
\end{proposition}

\begin{proof}
This is a consequence of Lemma~\ref{LemInvB-invB} together with Lemma~\ref{LemBrentBjornTypeB} and Definition~\ref{DefAdjaB}.
\end{proof}

An immediate consequence of Proposition~\ref{propTechInvB} is the following proposition.

\begin{proposition}\label{PropBAdja}
Let $\omega,\tau \in B_n$. Then, $\omega \lhd_R \tau$ if and only if there exists $a,b\in[\pm n]$ $B$-adjacent in $\omega$ such that $\Inv_B(\tau)=\Inv_B(\omega)\cup \{(a,b)\}$. 
\end{proposition}

Proposition~\ref{PropBAdja} implies the direct implication ($\Rightarrow$) of \eqref{EqInvBWeak}. 
Note that a proof of the converse implication of \eqref{EqInvBWeak}, which is of fundamental importance for our purpose, will follow from the results of Section~\ref{SectionTypeBProof} (see Corollary~\ref{CorolEquivInvWeakB}), and we will postpone till there. 

Let us now introduce a way to visualize $B$-inversion sets. First, note that the $B$-inversion set of any element of $B_n$ is a subset of $\{(a,b) \in [\pm n]^2\ | \ |a|<b \}$.
One can easily represent the set $\{(a,b) \in [\pm n]^2\ | \ |a|<b \}$ considering the \emph{shifted diagram} $\lambda^s_n$ of the partition $(2n-1,2n-3,\ldots,1)$, as depicted on Figure~\ref{Fig1}. The coordinates of each box can be read thanks to the circled integers. 

From now on, we identify $\lambda_n^s$ with the set $\{(a,b) \in [\pm n]^2\ | \ |a|<b \}$.

\begin{figure}[!h] 
\includegraphics[width=0.65\textwidth]{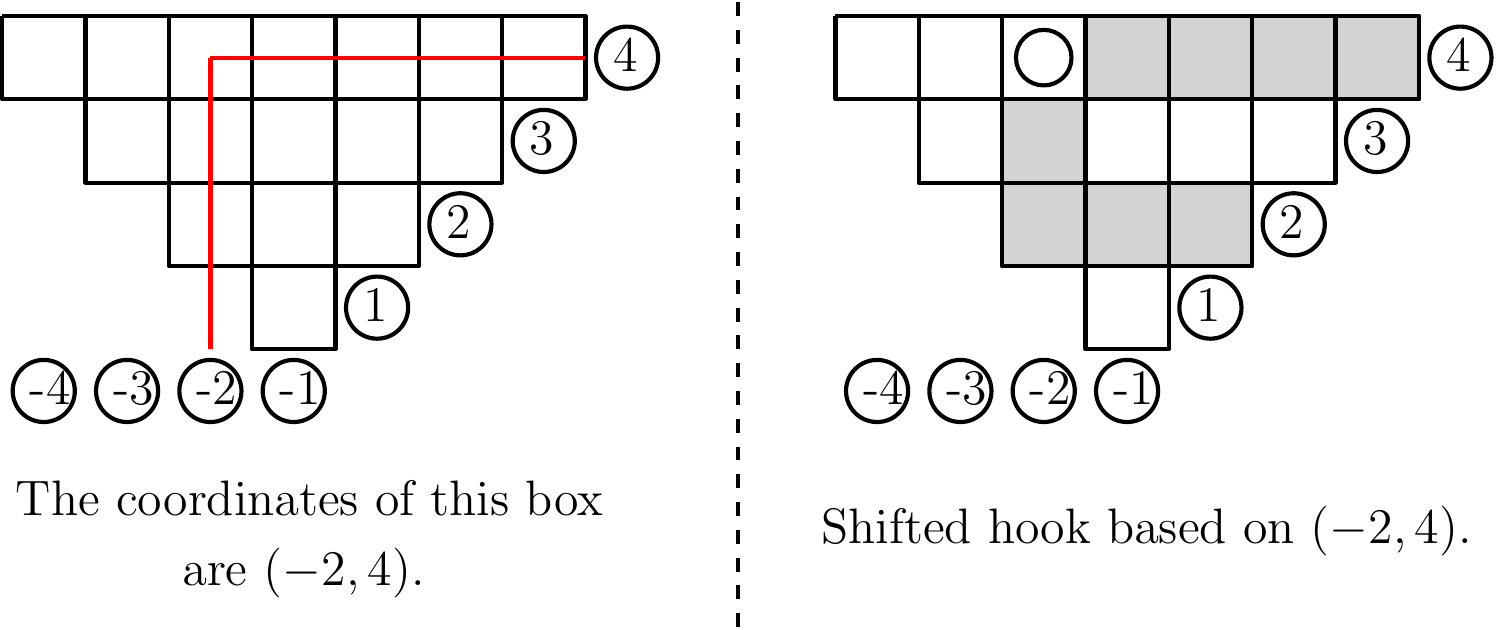}
\caption{}\label{Fig1}
\end{figure}

We now define a digraph structure $G$ on $\lambda_n^s$ (where the vertices are the boxes of the diagram), using the equivalent of hooks of Ferrers diagrams in the shifted case, namely \emph{shifted hooks} (as depicted on the right of Figure~\ref{Fig1}). The shifted hook based on $(a,b)$ in $\lambda_n^s$ is formally defined by 
\begin{align*}
\widetilde{H}(a,b):=\{ (a,b)\}\cup
\{(x,y)\in \lambda_n^s \ | \ \exists k \in &\mathbb{N} \ \text{such that} \ a<k<b \ \text{and} \\ 
(x,y)&=(k,b) \ {\rm or} \ (a,k) \ {\rm or} \ (-k,-a) \}.
\end{align*}
Following the methods of Section~\ref{SectSym}, we define a digraph structure $G$ on $\lambda_n^s$ by saying that there is an arc from $\mathfrak{c}$ to $\mathfrak{d}$ in $G$ if and only if $\mathfrak{c} \neq \mathfrak{d}$ and $\mathfrak{d}$ is in the shifted hook based on $\mathfrak{c}$. It appears that $G$ is acyclic and the out-degree of any box is an even number, so that the valuation $\theta(\mathfrak{c})=\frac{d^+(\mathfrak{c})}{2}$ is an OCV on $G$. Thus, $\B:=(G,\theta)$ is a valued digraph. Let us summarize this construction in a definition.

\begin{definition}
Let $G=(V,E)$ be the digraph such that 
\[V:=\lambda_n^s \ {\rm and} \ E:=\{(\mathfrak{c},\mathfrak{d}) \in (\lambda_n^s)^2 \ | \ \mathfrak{c} \neq \mathfrak{d} \ {\rm and} \ \mathfrak{d} \in \widetilde{H}(\mathfrak{c}) \}.\]
We denote by $\B=(G,\theta)$ the valued digraph such that for any $\mathfrak{c} \in \lambda_n^s$
\[ \theta(\mathfrak{c}):=\frac{d^+(\mathfrak{c})}{2}.\]
\end{definition}

One can easily check that the poset $(IS(\B),\subseteq)$ is isomorphic to the weak order on $(B_n,\leq_R)$ when $n=2$ or 3.
This situation is in fact general, as stated in the following theorem.

\begin{Theorem}\label{TheoTypB}
The posets $(B_n,\leq_R)$ and $(IS(\B),\subseteq)$ are isomorphic.
\end{Theorem}

The proof of this theorem follows the exact same pattern as the one of Theorem~\ref{SymCase}. However, many technical difficulties appear in the $B_n$ case, so that we detail completely the proofs in the following section.

\subsection{Proof of Theorem~\ref{TheoTypB}}\label{SectionTypeBProof}

In this section, we will show that $$IS(\B)=\{ \Inv_B(\omega) \ | \  \omega  \in B_n \},$$ which will imply Theorem~\ref{TheoTypB}. We follow the same method as in Section~\ref{SectionSymProof} and we divide our proof into three steps. First, for any $\omega \in B_n$ we define a statistic $d_{\omega}$ on $\lambda_n^s\setminus \Inv_B(\omega)$ (Definition~\ref{DefStatB}). Then by using the notion of adjacency (see \ref{DefAdjaB}) in $B_n$ we provide an alternative definition of $d_{\omega}$ (Lemma~\ref{LemAlterDefStatB}), leading to a combinatorial interpretation of the valuations appearing when one performs the peeling process on $\B$ (Proposition~\ref{TypeBPropTech}). Finally, we prove that $IS(\B)=\{ \Inv_B(\omega) \ | \ \omega \in B_n \}$ by using this combinatorial interpretation (Corollary~\ref{CorolPrincB}). Moreover, as a consequence we will obtain the converse implication of \eqref{EqInvBWeak} (Corollary~\ref{CorolEquivInvWeakB}), so that $(IS(\B),\subseteq)$ and $(B_n,\leq_R)$ are isomorphic. \\

We begin with the definition of the statistic $d_{\omega}$.

\begin{definition}\label{DefStatB}
Let $\omega \in B_n$ and $(a,b) \in \lambda_n^s \setminus \Inv_B(\omega)$. We define the statistic $d_{\omega}(a,b)$ as follows:
\begin{itemize}
\item if $|a|<b$, then $d_{\omega}(a,b):=|\{a<k<b \ | \ \omega^{-1}(a)<\omega^{-1}(k)<\omega^{-1}(b) \}|$;
\item if $-a=b$, then $d_{\omega}(a,-a):=|\{1 \leq k <-a \ | \ \omega^{-1}(a)<\omega^{-1}(k)<\omega^{-1}(-a) \}|$.
\end{itemize}
\end{definition}

The statistic $d_{\omega}$ admits the following characterization.

\begin{lemma}\label{LemAlterDefStatB}
Let $\omega \in B_n$, and $(a,b) \in \lambda_n^s \setminus \Inv_B(\omega)$, then we have $$d_{\omega}(a,b)=d_{Id}(a,b)-|\widetilde{H}(a,b) \cap \Inv_B(\omega)|.$$ 
\end{lemma}

\begin{proof}
Let $s_1\cdots s_q$ be a reduced decomposition of $\omega$ and denote by $\omega_i$ the signed permutation $s_1\cdots s_i$, $0 \leq i \leq q$. We will prove by induction on $i$ that the lemma is true for $\omega_i$.

The property is obviously true for $Id$. Let $i\geq 0$ be such that the property is true.
For the sake of clarity, let us denote by $\delta_j$ the integer $d_{\omega_j}(a,b)$. Since $s_1 \cdots s_q$ is reduced, we have $\sigma_i \lhd_R \sigma_{i+1}$, hence there exists a unique $(a_{i+1},b_{i+1})$ in $\Inv_B(\omega_{i+1}) \setminus \Inv_B(\omega_{i})$. We now show how one can deduce the value of $\delta_{i+1}$ from both $\delta_i$ and $(a_{i+1},b_{i+1})$. We split our study into three cases
\begin{itemize}
\item (Case $(a_{i+1},b_{i+1})=(a,k)$ with $a<k<b$ and $|a|\leq k$) $\omega_{i+1}$ is obtained from $\omega_i$ by swapping the positions of $a$ and $k$ and the positions of $-k$ and $-a$. Furthermore, $a$ and $k$ are adjacent in $\omega_i$ and $(a,b) \notin \Inv_B(\omega_i)$, so that we have 
\[ \omega_i=[\ldots, \ a \ , \ k \ ,\ldots, \ b \ ,\ldots]. \] 
We now distinguish two sub-cases.
\begin{itemize}
\item If $|a|<b$, then the full window notation of $\omega$ has one of the three following forms:
\begin{align*}
&[\ldots, \ a \ , \ k \ , \ldots , \ -k \ , \ -a \ , \ldots, \ b \ ,\ldots], \\
{\rm or} \ &[\ldots, \ a \ , \ k \ , \ldots , \ b \ ,\ldots, \ -k \ , \ -a \ , \ldots], \\
{\rm or} \ &[\ldots, \ -k \ , \ -a \ , \ldots , \ a \ , \ k \ , \ldots, \ b \ ,\ldots].
\end{align*}
Therefore, in all cases either both $-k$ and $-a$ are between $a$ and $b$ in $\omega_i$, or both $-k$ and $-a$ are not between $a$ and $b$. Hence, it is again the case in $\omega_{i+1}$, so we have $\delta_{i+1}=\delta_i-1$. Moreover, we also have $$|\widetilde{H}(a,b) \cap \Inv_B(\omega_{i+1})|=|\widetilde{H}(a,b) \cap \Inv_B(\omega_i)|+1.$$
\item If $b=-a$, then we have 
\[ \omega_i=[\ldots, \ a \ , \ k \ , \ldots, \ -k \ , \ -a \ ,\ldots ], \]
so that both $k$ and $-k$ lie between $a$ and $-a$ in $\omega_i$, and it is no longer the case in $\omega_{i+1}$. Hence, by definition of $d_{\omega}(a,-a)$, we have $\delta_{i+1}=\delta_i-1$. Furthermore, we also have $$|\widetilde{H}(a,b) \cap \Inv_B(\omega_{i+1})|=|\widetilde{H}(a,b) \cap \Inv_B(\omega_i)|+1.$$
\end{itemize}
\item (case $(a_{i+1},b_{i+1})=(k,b)$ or $(-k,-a)$ with $a<k<b$) using similar arguments as in the previous case, we show that $\delta_{i+1}=\delta_i-1$ (notice that the case $(a_{i+1},b_{i+1})=(-b,-k)$ cannot occur thanks to the condition $|a| \leq b$) and $$|\widetilde{H}(a,b) \cap \Inv_B(\omega_{i+1})|=|\widetilde{H}(a,b) \cap \Inv_B(\omega_i)|+1.$$
\item Otherwise, we have that both $a_{i+1}$ and $b_{i+1}$ either lie between $a$ and $b$ in $\omega_i$, or do not lie between $a$ and $b$ (and similarly for $-b_{i+1}$ and $-a_{i+1}$). Thus, it is still true in $\omega_{i+1}$, so that $\delta_{i+1}=\delta_i$ and $$|\widetilde{H}(a,b) \cap \Inv_B(\omega_{i+1})|=|\widetilde{H}(a,b) \cap \Inv_B(\omega_i)|.$$
\end{itemize}
By induction hypothesis $\omega_{i+1}$ satisfy the property, so that the lemma is proved.
\end{proof}

We now give a combinatorial interpretation of $\theta_A$ for some $A \in IS(\B)$

\begin{proposition}\label{TypeBPropTech}
Let $A \in IS(\A)$, if there exists $\omega \in B_n$ such that $A=\Inv_B(\omega)$, then for all $(a,b) \in \lambda_n^s \setminus \Inv_B(\omega)$, we have $\theta_A(a,b)=d_{\omega}(a,b)$.
\end{proposition}

\begin{proof}
Note that $\theta(a,b)=d_{Id}(a,b)$. Thus, by the definitions of the underlying digraph of $\B$ using shifted hooks and of the peeling process, and thanks to Lemma~\ref{LemAlterDefStatB}, the property follows.
\end{proof}

Proposition~\ref{TypeBPropTech} allows us to link the poset $(IS(\B),\subseteq)$ and the weak order on $B_n$, as it is shown in the next proposition.

\begin{proposition}\label{PropBAdjEras}
Let $A \in IS(\B)$. If there exists $\omega \in B_n$ such that $A=\Inv_B(\omega)$, then for all $(a,b)\in \lambda_n^s \setminus A$, we have that $(a,b)$ is erasable in $\B_A$ if and only if $a$ and $b$ are adjacent in $\omega$.
\end{proposition}

\begin{proof}
Let $(a,b) \in \lambda_n^s \setminus A$ and assume that $a$ and $b$ are adjacent in $\omega$. Our aim is to prove that $(a,b)$ is erasable in $\B_A$. 

First, note that $\theta_A(a,b)=d_{\omega}(a,b)=0$. Let $(c,d) \in \lambda_n^s \setminus A$ be such that there is an arc from $(c,d)$ to $(a,b)$. 
We will prove that $\theta_A(c,d) \neq 0$. Equivalently, we will show that $d_{\omega}(c,d) \neq 0$. By definition of the underlying digraph of $\B$, we have only three cases which we now detail.
\begin{itemize}
\item ($(a,b)=(c,p)$ such that $|c|\leq p < d$). Since $a$ and $b$ are adjacent in $\omega$, we have that $c$ and $p$ are adjacent in $\omega$. Moreover, we have $(c,d) \notin A=\Inv_B(\omega)$, so that $c$ is on the left of $d$ in the window notation of $\omega$. It follows that we have
\[ \omega= [ \ldots, \ c \ , \ p \ , \ldots , \ d \ , \ldots ].  \]
However, we have $c < p<d$ by hypothesis, hence $d_{\omega}(c,d) \geq 1$.
\item ($(a,b)=(q,d)$ with $c<q<d$). We have that $q$ and $d$ are adjacent in $\omega$. Moreover, we have $(c,d) \notin\Inv_B(\omega)$, so that we have
\[ \omega= [ \ldots, \ c \ ,  \ldots , \ q \  , \ d \ , \ldots ]. \]
Nevertheless, we have $c<q<d$ by hypothesis, hence $d_{\omega}(c,d) \geq 1$.
\item ($(a,b)=(k,-c)$ with $c \leq k < -c$). First, note that we have $c<-k \leq -c$. Moreover, we have $(c,d) \in \lambda_n^s$, so that $-c \leq |c| \leq d$. We thus have $c<-k \leq -c \leq d$. Assume by contradiction that $-k=d$, then $-c=d$, hence we have
\[ (a,b)=(-d,d)=(k,-c)=(c,d).\]
Consequently, there is an arc from $(a,b)$ to $(a,b)$ in the underlying graph of $\B$, and this is absurd. Therefore, we have $c<-k<d$. 

Let us now show that $-k$ lies between $c$ and $d$ in $\omega$. By hypothesis, we have
\[ \omega=[ \ldots, \ k \  , \ -c \  , \ldots ], \]
but $\omega$ is a signed permutation, so that we have
\[ \omega=[ \ldots, \ c \  , \ -k \  , \ldots ]. \]
However, $(c,d) \notin \Inv_B(\omega)$, hence we have 
\[ \omega=[ \ldots, \ c \  , \ -k \  , \ldots , \ d \ , \ldots ]. \]
Therefore, if $-c \neq d$, then we have $d_{\omega}(c,d) \geq 1$. If $-c=d$, then we have 
\[ \omega=[ \ldots, \ c \  , \ -k \  , \ldots , \ k \ , \ -c \ , \ldots ], \]
and we also have $c<-k<k<-c$, so that $d_{\omega}(c,d) \geq 1$
\end{itemize}
In all cases, we have $d_{\omega}(c,d) \geq 1$, but $\theta(c,d)=d_{\omega}(c,d)$ by Proposition~\ref{TypeBPropTech}, hence $\theta(c,d) \geq 1$. Thus, we just proved that for all box $\mathfrak{c} \in \lambda_n^s\setminus A$, if there is an arc from $\mathfrak{c}$ to $(a,b)$, then $\theta(\mathfrak{c}) \geq 1$. Consequently, $(a,b)$ is erasable in $\B_A$.

Let us now prove the converse. Let $(a,b) \in \lambda_n^s \setminus A$ be erasable in $\B_A$ and assume by contradiction that $a$ and $b$ are not adjacent in $\omega$. We divide the study into two cases. 
\begin{itemize}
\item (Case $a=-b$) Since $-b$ and $b$ are not adjacent in $\omega$, there exists $k$ lying between $-b$ and $b$ in $\omega$. By symmetry, both $k$ and $-k$ lie between $-b$ and $b$, thus we can suppose that $k>0$. Furthermore, $d_{\omega}(-b,b)=0$, so that $k>b$. Let us consider $p>b$ minimal lying between $a$ and $b$ and let $q$ be an integer lying between $-b$ and $p$ in $\omega$ (if such a $q$ exists). Then, we have 
\[ \omega=[ \ldots, \ -b \ , \ldots , \ q \ , \ldots ,  \ p \ , \ldots ,  \ b \ , \ldots ], \]
so that $q$ is between $-b$ and $b$ in $\omega$. Moreover, we have the following facts:
\begin{itemize}
\item by minimality of $p$, $q \notin \{b,b+1,\ldots,p\}$;
\item since $d_{\omega}(-b,b)=0$, $q \notin \{-b,-b+1,\ldots,b\}.$
\end{itemize}
Consequently, $q \notin \{-b,-b+1,\ldots,p\}$, hence $d_{\omega}(-b,p)=0=\theta_A(-b,p)$. However, there is an arc from $(-b,p)$ to $(-b,b)$, and this is a contradiction since $(-b,b)$ is erasable. 
\item (Case $|a|<b$) There exists $k$ lying between $a$ and $b$ such that either $k>b$ or $k<a$. In the first case, similar arguments as in the previous case lead to a contradiction with the fact that $(a,b)$ is erasable. In the second case, we consider $p<a$ maximal lying between $a$ and $b$. We have the following two sub-cases.
\begin{itemize}
\item If $-b \leq p < a$, then for each $q$ between $p$ and $b$ in $\omega$, we have either $q<p$ by maximality, or $q>b$ because $d_{\omega}(a,b)=0$. Thus, $d_{\omega}(p,b)=0$, so that $\theta(p,b)=0$. but there is an arc from $(p,b)$ to $(a,b)$, hence it contradicts the fact that $(a,b)$ is erasable.
\item If $p<-b$, then we will prove that $d_{\omega}(-b,-p)=0$. For that purpose, assume by contradiction that $d_{\omega}(-b,-p)=0$. Thus,
 there exists $q$ between $-b$ and $-p$ in $\omega$ such that $-b<q<-p$. Then, we have $p<-q<b$. Moreover, since $\omega$ is a signed permutation we have
\begin{equation}\label{EquationIntermedTypeB}
\omega=[ \ldots, \ p \ , \ldots , \ -q \ , \ldots ,  \ b \ , \ldots ],
\end{equation}
but $p$ is between $a$ and $b$ in $\omega$, hence we have
\[ \omega=[ \ldots, \ a \ , \ldots , \ -q \ , \ldots ,  \ b \ , \ldots ]. \]
Therefore, we have $p<-q<b$ and by maximality of $p$, we have $a \leq -q$. Since $p$ is between $a$ and $b$ in $\omega$, thanks to \ref{EquationIntermedTypeB}, we have $a \neq -q$. Eventually, we have $a<-q<b$, and this is absurd since $d_{\omega}(a,b)=0$. Consequently, we have $d_{\omega}(-b,-p)=0$, so that $\theta(-b,-p)=0$. But there is an arc from $(-b,-p)$ to $(a,b)$, hence it contradicts the fact that $(a,b)$ is erasable.
\end{itemize}
\end{itemize}
Finally, in all cases we obtain a contradiction. Thus, $a$ and $b$ are adjacent in $\omega$ and this concludes the proof.
\end{proof}

With Proposition~\ref{PropBAdjEras}, one can prove the following result using exactly the same method as in the proof of Corollary~\ref{CorolFinalA}.

\begin{corollary}\label{CorolPrincB}
$IS(\B)=\{\Inv_B(\omega)\ | \ \omega \in B_n \}$.
\end{corollary}

This result has the following important consequence (which gives the converse direction of \eqref{EqInvBWeak}).

\begin{corollary}\label{CorolEquivInvWeakB}
Let $\sigma,\omega\in B_n$, then $\sigma \leq_R \omega$ if and only if $\Inv_B(\sigma)\subseteq \Inv_B(\omega)$.
\end{corollary}

\begin{proof}
The direct direction is given by \eqref{EqInvBWeak}, and we now prove the converse. Assume that $\Inv_B(\sigma) \subseteq \Inv_B(\omega)$, then there exist $L=[z_1,\ldots,z_{n^2}]\in PS(\B)$ and $p\leq q$ two integers such that $\Inv(\sigma)=\{z_1,\ldots,z_p\}$ and $\Inv(\omega)=\{z_1,\ldots,z_q\}$. Then thanks to Corollary~\ref{CorolPrincB} and Proposition~\ref{PropBAdja}, there exist $\sigma_1,\ldots,\sigma_k$ such that $\sigma=\sigma_1 \lhd_R \sigma_2 \lhd_R \ldots \lhd_R \sigma_k=\omega$, and this ends the proof.
\end{proof}

Consequently, thanks to Corollary~\ref{CorolPrincB} the posets $(IS(\B),\subseteq)$ and $(B_n,\leq_R)$ are isomorphic.
This concludes the proof of Theorem~\ref{TheoTypB}.

\subsection{Weak order on $\widetilde{A}_n$}\label{SectAffineA}
Recall that $\widetilde{A}_n$ is the Coxeter Group with generating set $S=\{ s_1,\ldots,s_n\}$, and with Coxeter matrix given by $m_{s_i s_{j}}=3$ if $j=i+1$ (where the indices are taken modulo $n$), and $m_{s_i s_j}=2$ otherwise. This group can be seen as the group of the \emph{affine permutations}, that is, the group of all the bijections $\sigma:\mathbb{Z}\mapsto \mathbb{Z}$ such that:
\begin{enumerate}
\item for all $k$ and $q$ in $\mathbb{Z}$, $\sigma(q+kn)=\sigma(q)+kn$;
\item $\sigma(1)+\sigma(2)+\ldots + \sigma(n)=\frac{n(n+1)}{2}$.
\end{enumerate}
Thanks to this interpretation, we identify $s_i$ with the affine permutation swapping positions of the integers $i+kn$ and $i+1+kn$, for all $k \in \mathbb{Z}$.

We are going to follow the same method of the previous sections. In order to find a candidate of valued digraph, we consider a notion of $\widetilde{A}$-inversion set adapted to the case of $\widetilde{A_n}$ (see Definition~\ref{DefAffInv}). After we checked that this notion is effectively related to the weak order on $\widetilde{A_n}$ (see Property~\ref{EqInvAffineWeak}), we propose a graphical interpretation of $\widetilde{A}$-inversion sets using \emph{cylindrical diagrams}. Once again, this representation carries a natural notion of hooks, called \emph{cylindrical hooks}, which leads us to define a digraph structure on a cylindrical diagram as in Sections~\ref{SectSym} and \ref{SectB}. Then, we define a valued digraph using the resulting digraph, and we check that the obtained lattice is indeed isomorphic to $(\widetilde{A_n},\leq_R)$. 

\begin{Remark}
We point out that, as in Section~\ref{SectB}, the notion of $\widetilde{A}$-inversion set we use here does not come from a root system of $\widetilde{A_n}$. This choice gives the same benefits (a ``permutation point of view'' on the weak order) and disadvantages (we will have to relate $\widetilde{A}$-inversion sets to the weak order by ourselves) as in Section~\ref{SectB}. Fortunately, once again most of technical points have already been studied in \cite{BB}.
\end{Remark}

We associate to each affine permutation $\sigma$ an $\widetilde{A}$-inversion set, define as follows:
\begin{equation}\label{DefAffInv}
\Inv_{\widetilde{A}}(\sigma):=\{(a,b)\in [n]\times \mathbb{N}^* \ |  \ a<b, \ {\rm and} \ \sigma^{-1}(a)>\sigma^{-1}(b) \},
\end{equation}

\begin{definition}
(see~\cite[Eq. (8.30)]{BB}).
Let $\sigma \in \widetilde{A_n}$, the $\widetilde{A}$-inversion number of $\sigma$ is the quantity
\begin{equation*}
\inv_{\widetilde{A}}(\sigma):=|\{(a,b)\in[n]\times \mathbb{N}^*\ |\ a<b \ {\rm and} \ \sigma(a)>\sigma(b) \}|
\end{equation*}
\end{definition}

Note that we clearly have $\inv_{\widetilde{A}}(\sigma)=|\Inv_{\widetilde{A}}(\sigma)|$ for all $\sigma \in \widetilde{A_n}$.

\begin{lemma}[see \cite{BB}, Eq. (8.34) p. 262]\label{LemBrentBjornTypeAffine}
Let $\sigma \in \widetilde{A_n}$ and $i\in [n]$, we have
\begin{equation*}
\inv_{\widetilde{A}}(\sigma s_i)=\left\{
\begin{aligned}
\inv_{\widetilde{A}}(\sigma)+1, \ & \ \text{if} \ \sigma(i)<\sigma(i+1), \\
\inv_{\widetilde{A}}(\sigma)-1, \ & \ \text{if} \ \sigma(i)>\sigma(i+1).
\end{aligned}
\right.
\end{equation*}
\end{lemma}

We now begin to prove that $\widetilde{A}$-inversion sets can be used to study $(\widetilde{A_n},\leq_R)$. That is, we will show that we have 
\begin{equation}\label{EqInvAffineWeak}
\text{for all} \ \sigma,\omega \in \widetilde{A_n}, \ \text{if} \ \sigma \leq_R \omega, \ \text{then} \ \Inv_{\widetilde{A}}(\sigma) \subseteq \Inv_{\widetilde{A}}(\omega).
\end{equation}

We start with defining the equivalent of the notion of adjacency in type $\widetilde{A}$.

\begin{definition}\label{DefAffInv}
Let $\sigma \in \widetilde{A_n}$ and $(a,b) \in [n] \times \mathbb{N}^*$. We say that $a$ and $b$ are \emph{$\widetilde{A}$-adjacent} in $\sigma$ if and only if $a<b$ and $\sigma^{-1}(a)=\sigma^{-1}(b)-1$.
\end{definition}

We are now able to state the lemma which connects $\widetilde{A}$-inversion sets to $(\widetilde{A_n},\leq_R)$.

\begin{lemma}\label{LemAffAdja-Cover}
Let $\sigma \in \widetilde{A_n}$ and $1 \leq j \leq n$. Then, there exists $(a,b)$ such that $1 \leq a \leq n$, $a<b$ and $\sigma s_j$ is obtained from $\sigma$ by swapping positions of the integers $a+kn$ and $b+kn$ for all $k \in \mathbb{Z}$, and we have two possibilities:
\begin{itemize}
\item if $a$ and $b$ are adjacent in $\sigma$, then $\ell(\sigma s_j)=\ell(\sigma)+1$ and $\Inv_{\widetilde{A}}(\sigma s_j)=\Inv_{\widetilde{A}}(\sigma) \cup \{(a,b)\}$;
\item if $a$ and $b$ are not adjacent in $\sigma$, then $\ell(\sigma s_j)= \ell(\sigma)-1$ and $\Inv_{\widetilde{A}}(\sigma s_j)=\Inv_{\widetilde{A}}(\sigma) \setminus \{(a,b)\}$.
\end{itemize}
\end{lemma}

\begin{proof}
This is an immediate translation of the results in Lemma~\ref{LemBrentBjornTypeAffine} in terms of $\widetilde{A}$-adjacency.
\end{proof}

An immediate consequence of Lemma~\ref{LemAffAdja-Cover} is that \eqref{EqInvAffineWeak} holds. 

\begin{Remark}
As in Section~\ref{SectB}, note that the converse implication holds, and it is also a by-product of the following results.
\end{Remark}

We now introduce a convenient way to represent $\widetilde{A}$-inversion sets. First, note that for all $\sigma \in \widetilde{A_n}$ and for all $a,b \in \mathbb{Z}$ such that $1 \leq a \leq n$ and $b\equiv a \pmod n$, since $\sigma$ is an affine permutation we have $(a,b) \notin \Inv_{\widetilde{A}}(\sigma)$.
Thus, each $\widetilde{A}$-inversion is a subset of 
\[ \{(a,b)\in \mathbb{N}^2 \ | \ 1\leq a \leq n, \ b \not \equiv a \pmod n, \ a<b  \}.\] 
This set can be represented by a diagram, which we denote by $\lambda_n^{cyl}$, as depicted in Figure~\ref{FigAffineType1}.
\begin{figure}[!h]
\includegraphics[width=0.65\textwidth]{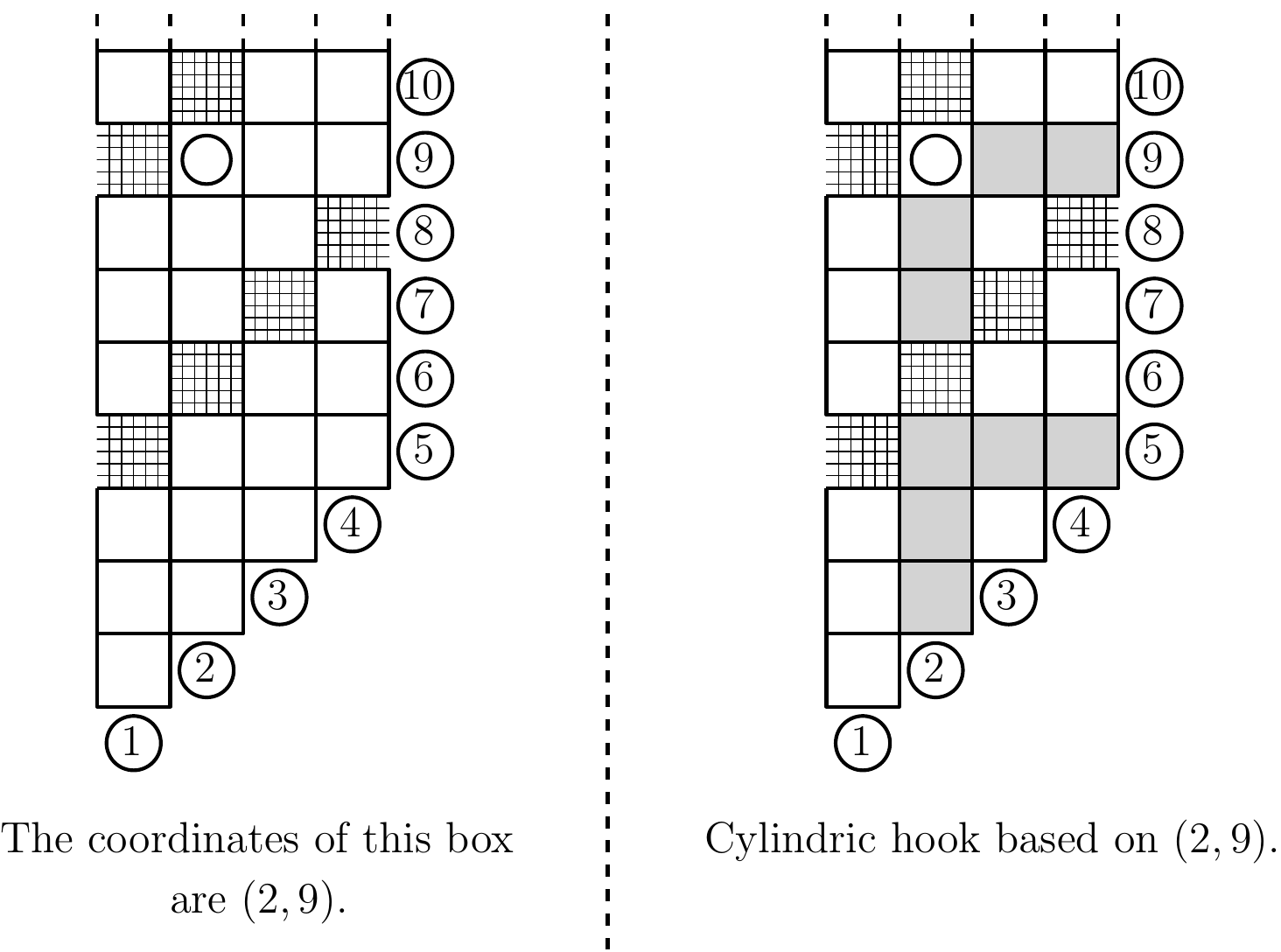}
\caption{Diagram $\lambda_4^{cyl}$}\label{FigAffineType1}
\end{figure}

From now on, we identify $\lambda_n^{cyl}$ with the set $\{(a,b)\in \mathbb{N}^2 \ | \ 1\leq a \leq n, \ b \not \equiv a \pmod n, \ a<b  \}.$
This diagram $\lambda_n^{cyl}$ can be thought as an infinite version of the diagram associated with the symmetric group rolled around a cylinder. With this point of view, $\lambda_n^{cyl}$ naturally carries a notion of hooks, which we call  \emph{cylindrical hooks}, as depicted on the right of Figure~\ref{FigAffineType1}. More formally, for all $(a,b) \in \lambda_n^{cyl}$, the cylindrical hook based on $(a,b)$ is the subset $H^{cyl}(a,b)$ of $\lambda_n^{cyl}$ defined by: 
\begin{equation*}
H^{cyl}(a,b):= \ \{ (a,k)\in \lambda_n^{cyl} \ | \ a<k<b \} \ 
\cup \ \left( \bigcup_{\substack{q \equiv b\! \! \pmod n \\q\leq b }}\{(k,q)\in \lambda_n^{cyl} \ | \ a<k<b \} \right)
\end{equation*}

Consequently, we can define a digraph structure $G$ on $\lambda_n^{cyl}$ using cylindrical hooks. That is, for all $\mathfrak{c},\mathfrak{d} \in \lambda_n^{cyl}$, there is an arc from $\mathfrak{c}$ to $\mathfrak{d}$ in $G$ if and only if $\mathfrak{c}\neq \mathfrak{d}$ and $\mathfrak{d}\in H^{cyl}(\mathfrak{c})$. 
Notice that the out-degree of a box of $\lambda_n^{cyl}$ is generally not an even number, so that we cannot define the valuation as in the previous sections. Nevertheless, after some tests it appears that the valuation $\theta$ defined for all $(a,b)\in \lambda_n^{cyl}$ by
\[ \theta(a,b):=|\{(a,k) \in \lambda_n^{cyl} \ | \ a<k<b \}|,\]
which is just the number of boxes which are below $(a,b)$ in the graphical representation of $\lambda_n^{cyl}$, seems to lead to the expected description of the weak order on $\widetilde{A_n}$. Before moving to the proof that is is indeed the case, let us summarize this construction in a definition.

\begin{definition}
Let $G=(V,E)$ be the digraph defined by 
\[V:=\lambda_n^{cyl} \ {\rm and} \ E:=\{(\mathfrak{c},\mathfrak{d})\in (\lambda_n^{cyl})^2 \ | \ \mathfrak{c}\neq \mathfrak{d} \ {\rm and} \ \mathfrak{d}\in H^{cyl}(\mathfrak{c}) \}.\]
We denote by $\widetilde{\A}=(G,\theta)$ the valued digraph such that for all $(a,b) \in \lambda_n^{cyl}$,
\[ \theta(a,b)=|\{(a,k) \in \lambda_n^{cyl} \ | \ a<k<b \}|.\]
\end{definition}

Our aim is now to prove that $(IS(\widetilde{\A}),\subseteq)$ is isomorphic to $(\widetilde{A_n},\leq_R)$. This can be done following exactly the same method as in Section~\ref{SectionTypeBProof}, and we refer the reader to the introduction of Section~\ref{SectionTypeBProof} for the detail of the different steps.

We first define the statistic on the affine permutations, which will lead us to the combinatorial interpretation of the valuations appearing when one perform the peeling process on $\widetilde{\A}$.

\begin{definition}
Let $\sigma \in \widetilde{A_n}$ and $(a,b) \in \lambda_n^{cyl}\setminus \Inv_{\widetilde{A}}(\sigma)$. We set
\begin{align*}
d_{\sigma}(a,b)&:=|\{a<k<b \ | \ k \not\equiv a \!\! \pmod n, \ \sigma^{-1}(a) \leq \sigma^{-1}(k) \leq \sigma^{-1}(b) \}|.
\end{align*}
\end{definition}

We then have the following alternative definition of the statistic $d_{\sigma}$

\begin{lemma}\label{LemTechAffin}
For all $\sigma \in \widetilde{A_n}$ and $(a,b) \in \lambda_n^{cyl} \setminus \Inv_{\widetilde{A}}(\sigma)$, we have 
\[d_{\sigma}(a,b)=d_{Id}(a,b)-|H^{cyl}(a,b)\cap \Inv_{\widetilde{A}}(\sigma)|.\]
\end{lemma}

\begin{proof}
The proof is similar as the one of Lemma~\ref{LemSymTech}.
\end{proof}

Thanks to Lemma~\ref{LemTechAffin}, we have the following proposition.

\begin{proposition}
Let $A \in IS(\widetilde{A})$, if there exists $\sigma \in \widetilde{A_n}$ such that $A=\Inv_{\widetilde{A}}(\sigma)$, then for all $(a,b) \in \lambda_n^{cyl}\setminus A$ we have $\theta_A(a,b)=d_{\sigma}(a,b)$.
\end{proposition}

\begin{proof}
The proof is similar as the one of Proposition~\ref{PropSymCombinatInt}.
\end{proof}

We can now state and prove the main proposition of this section.

\begin{proposition}
Let $A \in IS(\widetilde{A})$, if there exists $\sigma \in \widetilde{A_n}$ such that $A=\Inv_{\widetilde{A}}(\sigma)$, then $(a,b)$ is erasable in $\widetilde{A}_A$ if and only if $a$ and $b$ are $\widetilde{A}$-adjacent in $\sigma$.
\end{proposition}

\begin{proof}
Once again, the proof is similar to that of Proposition~\ref{TechSym}.
\end{proof}

Eventually, we have the following three corollaries that conclude this section.

\begin{corollary}
$IS(\widetilde{A})=\{\Inv_{\widetilde{A}}(\sigma) \ | \ \sigma \in \widetilde{A_n} \}$.
\end{corollary}

\begin{corollary}
Let $\sigma,\omega \in \widetilde{A_n}$. Then,
$\sigma \leq_R \omega, \ \text{if and only if} \ \Inv_{\widetilde{A}}(\sigma) \subseteq \Inv_{\widetilde{A}}(\omega)$.

\end{corollary}

\begin{corollary}\label{TheoAffine}
The two posets $(IS(\widetilde{\A}),\subseteq)$ and $(\widetilde{A_n},\leq_R)$ are isomorphic.
\end{corollary}

\subsection{Flag Weak Order on $\mathbb{Z} _r \wr S_n$}\label{SectionFlagWeak}

In this section, we consider an order on $G(r,n):=\mathbb{Z} _r \wr S_n$ (introduced by Adin, Brenti and Roichman in \cite{ABR}), called the \emph{flag weak order}, that generalizes the weak order on the symmetric group. In order to define this new poset, let us first introduce some notations and definitions.
We denote by $\mathbb{Z}_r$ the (additive) cyclic group of order $r$ and by $G(r,n)$ the group
 $$G(r,n):=\{((c_1,\ldots,c_n),\sigma) \ | \ c_i \in [r], \ \sigma \in S_n  \}$$ 
with the group operation given by $$((c_1,\ldots,c_n),\sigma).((d_1,\ldots,d_n),\omega)=((c_{\omega(1)}+d_1,\ldots,c_{\omega(n)}+d_n),\sigma \omega),$$ 
where the sums $c_{\omega(i)}+d_i$ are taken modulo $r$.
This group is usually called the group of \emph{$r$-colored permutations}, i.e. bijections $g$ of the set $\mathbb{Z}_r \times \{1,\ldots,n \}$ onto itself such that: $$g(c,i)=(d,j) \Longrightarrow g(c+c',i)=(d+c',j).$$
Note that the group $G(r,n)$ can also be viewed as a \emph{complex reflection group}. However, once again it is the ``permutation point of view" on $G(r,n)$ which will allow us to apply our theory here. 

Before moving to the definition of the flag weak order, we introduce some useful notations taken from \cite{ABR}.

\begin{definition}\label{DefFlagStat}
Let $\pi=((c_1,\ldots,c_n),\sigma)$ be in $G(r,n)$, we define:
\begin{enumerate}
\item $|\pi|=\sigma$;
\item ${\rm Inv}(\pi)={\rm Inv}(|\pi|)$;
\item $n(\pi)=\sum_{i} c_i$;
\item ${\rm finv}(\pi)=r.|{\rm Inv}(\pi)|+n(\pi)$, called the \emph{flag inversion number} of $\pi$.
\end{enumerate}
\end{definition}

Let us now present the philosophy behind the definition of the flag weak order, by first recalling the definition of the weak order on the symmetric group. The definition of $(S_n,\leq_R)$ can be decomposed into two distinct steps:
\begin{itemize}
\item first, we consider a specific set $S$ of generator of $S_n$ (here the simple transpositions);
\item then, we consider a statistic $\ell$ on $S_n$ (here the length) and we define the weak order to be the reflexive and transitive closure of the relation defined by: $\text{for all} \ \sigma,\omega \in S_n,$
\[ \sigma \lhd_R \omega \Longleftrightarrow \exists s \in S \ \text{such that} \ \ \omega=\sigma s \ \ {\rm and} \ \ \ell(\tau)<\ell(\pi).\]
\end{itemize}
The flag weak order is defined by following a similar pattern: 
\begin{itemize}
\item first, we define a special generating set of $G(r,n)$, denoted by $A\cup B$;
\item then, we define the flag weak order to be the reflexive and transitive closure of the relation defined by: $\text{for all} \ \pi,\tau \in G(r,n),$
$$ \pi \lhd_f \tau \Longleftrightarrow \exists s \in A \cup B \ \text{such that} \ \tau=\pi s \ \ {\rm and} \ \ {\rm finv}(\tau)<{\rm finv}(\pi).$$
\end{itemize}
As one can notice, the only difference with the definition of $(S_n,\leq_R)$ is that we swapped the length with the flag inversion number (see Definition~\ref{DefFlagStat}).
Let us now formalize this construction in a definition.

\begin{definition}[Flag weak order, see~\cite{ABR}]
We denote by $A$ and $B$ the two subsets of $G(r,n)$ defined by
\begin{align*}
A&=\left\{a_i\in G(r,n) \ | \ i \in [n-1], \ a_i=((\delta_{i1},\ldots,\delta_{ii},\ldots,\delta_{in}),s_i) \right\}, \ \text{and} \\
 B&=\left\{b_i\in G(r,n) \ | \ i \in [n], \ b_i=((\delta_{i1},\ldots,\delta_{ii},\ldots,\delta_{in}),Id) \right\},
\end{align*}
where $s_i$ is the $i$-th elementary transposition of the symmetric group and $\delta_{ij}=1$ if $i=j$, and 0 otherwise.
The \emph{flag weak order} $\leq _f$ on $G(r,n)$ is the reflexive and transitive closure of the relation $\lhd_f$ defined by:
$$\forall \pi, \tau \in G(r,n), \ \pi \lhd_f \tau \Longleftrightarrow \exists s \in A \cup B \ \text{such that} \ \tau=\pi s \ \ {\rm and} \ \ {\rm finv}(\tau)<{\rm finv}(\pi).$$
\end{definition}

The following lemma provide a complete description of covering relations in the flag weak order.

\begin{lemma}[\cite{ABR}, Prop. 7.4]\label{CaracCover}
Let $\pi=((c_1,\ldots,c_n),\sigma) \in G(r,n)$ and $s \in A \cup B$. Then, $\pi s$ covers $\pi$ in the flag weak order if and only if one of the two following situations occur:
\begin{enumerate}
\item there exists $1 \leq i \leq n$ such that $s=b_i \in B$ and $c_i \neq r-1$;
\item there exists $1 \leq i \leq n-1$ such that $s=a_i \in A$, $c_{i+1}=r-1$ and $\sigma(i)<\sigma(i+1)$.
\end{enumerate}
\end{lemma}

We now have enough general informations about the flag weak order to propose a valued digraph which describes the flag weak order. The key point leading us to the construction of this valued digraph is that the elements of $A$ ``look like'' the simple transpositions of $S_n$. That is, they act on $r$-coloured permutations as simple transpositions act on permutation, by swapping the positions of two adjacent entries. 
Furthermore, one can note that the function 
\[
\begin{array}{ccccc}
\phi &:& S_n &\longrightarrow& G(r,n) \\
&&\sigma &\longmapsto& ((0,\ldots,0),\sigma)
\end{array}
\]
is an injective poset morphism from $(S_n,\leq_R)$ to $(G(r,n),\leq_f)$. Thus, the flag weak order contains a sub-poset isomorphic to the weak order on $S_n$.
Using this facts as hints and after some ``guess and try'' tests on the example of $G(2,4)$, the author found out a candidate of valued digraph, which is depicted on Figure~\ref{FigExFlag}.
\begin{figure}[!h]
\includegraphics[width=0.5\textwidth]{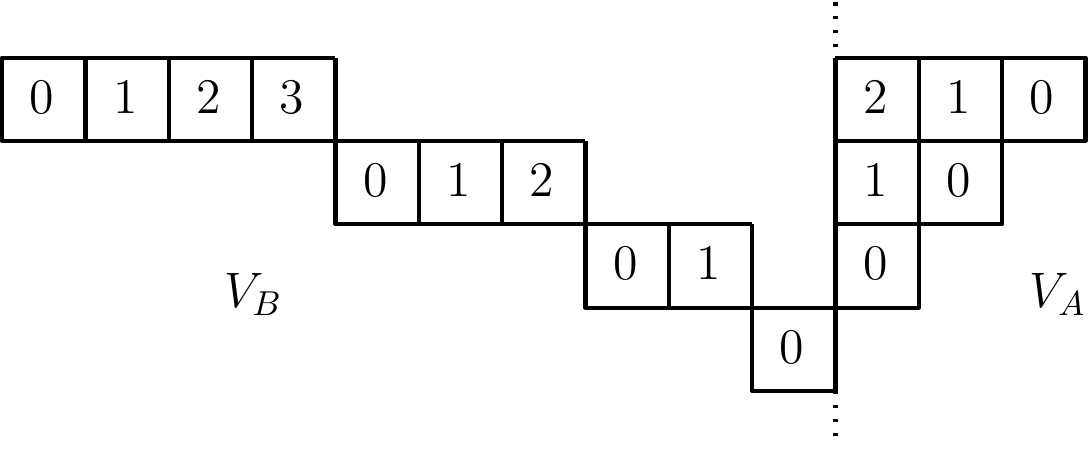}
\caption{The candidate of valued digraph for $(G(2,4),\leq_f)$}\label{FigExFlag}
\end{figure}
Once again, the digraph structure of this diagram is given implicitly, using a suitable notion of hook. A bit more formally, we say that for all boxes $\mathfrak{c}$ and $\mathfrak{d}$ in this diagram, there is an arc from $\mathfrak{c}$ to $\mathfrak{d}$ if and only if $\mathfrak{c}\neq \mathfrak{d}$ and $\mathfrak{d}$ is either in the same row and on the right of $\mathfrak{c}$, or in the same column and below $\mathfrak{c}$. With this definition, one can check that the resulting poset is indeed isomorphic to $(G(2,4),\leq_f)$.

Let us now generalize and formalize this construction to the case of $r$ and $n$ arbitrary, by first defining the diagram.

\begin{definition}
We set $\lambda_{r,n}:=V_{A,n} \cup V_{B,r,n}$ where
 \begin{align*}
V_{A,n}&:=\left\{(a,b) \in [n]^2 \ | \ a<b \right\}, \\ 
\text{and} \ V_{B,r,n}&:=\left\{ (a,b) \in \mathbb{Z} \times [n] \ | \ -b(r-1) \leq a \leq -1 \right\}.
\end{align*}
\end{definition}

\begin{figure}[!h]
\includegraphics[width=0.52\textwidth]{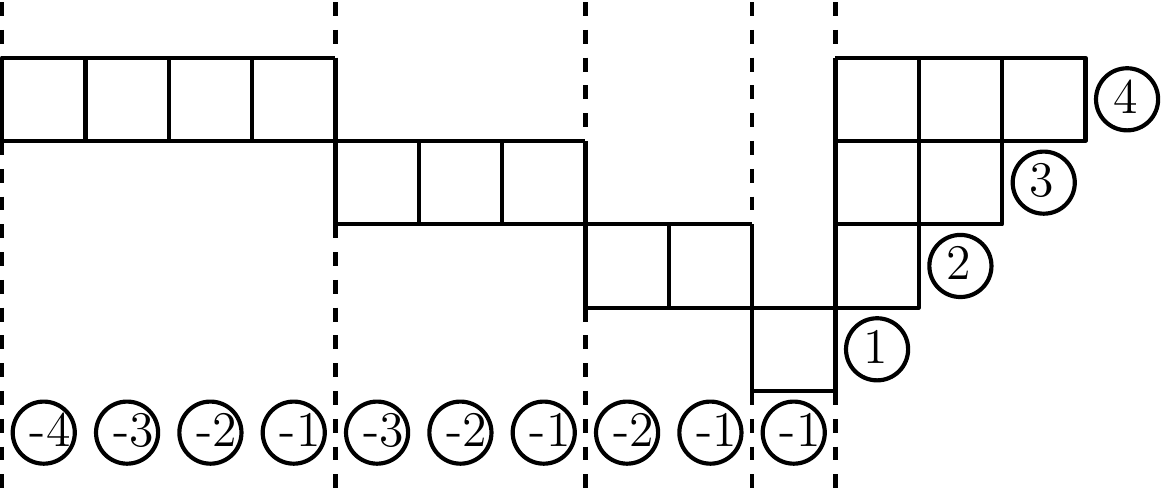}
\caption{Graphical representation of $\lambda_{2,4}$ with its coordinates}\label{FigExFlag2}
\end{figure}

Note that we have $V_{A,n}=\lambda_n$ (see Section~\ref{SectSym}), and we will sometimes use this notation. Let us now define the notion of hook associated with the diagram $\lambda_{r,n}$, being suggested by its graphical representation (see Figure~\ref{FigExFlag2}).

\begin{definition}
Let $(a,b) \in \lambda_{r,n}$, we denote by $H_f(a,b)$ the subset of $\lambda_{r,n}$ defined by:
\begin{itemize}
\item if $(a,b) \in V_{A,n},$ then 
\[
H_f(a,b):=\{(x,y) \in \lambda_{r,n} \ | \ \exists k\in \mathbb{N}, \ a<k<b, \ (x,y)=(a,k) \ \text{or} \ (k,b) \};
\]
\item if $(a,b) \in V_{B,r,n}$, then
\[
H_f(a,b):=\{(x,y) \in \lambda_{r,n} \ | \ a<x \ \text{and} \ y=b \}.
\]
\end{itemize}
\end{definition}

Eventually, we can now define the valued digraph.

\begin{definition}
Let $G=(V,E)$ be the digraph defined by
\[ V:=\lambda_{r,n} \ \text{and} \ E:=\{(\mathfrak{c},\mathfrak{d})\in \lambda_{r,n}^2 \ | \ \mathfrak{c}\neq \mathfrak{d} \ \text{and} \ \mathfrak{d} \in H_f(\mathfrak{c}\}. \]
We denote by $\G(r,n):=(G,\theta)$ the valued digraph such that for all $(a,b) \in \lambda_{r,n}$
\begin{equation*}
 \theta(a,b):=\left\{
\begin{aligned}
&b-a-1 && \text{if} \ (a,b) \in V_{A,n}, \\
&b+\left\lfloor\frac{a}{r-1} \right\rfloor && \text{if} \ (a,b)\in V_{B,r,n}.
\end{aligned}
\right.
\end{equation*}
\end{definition}

\begin{example}
We represent on Figure~\ref{FigFlag3} the valued digraph $\G(2,4)$. As one can see, this is exactly the valued digraph depicted on Figure~\ref{FigExFlag}.
\begin{figure}[!h]
\includegraphics[width=0.5\textwidth]{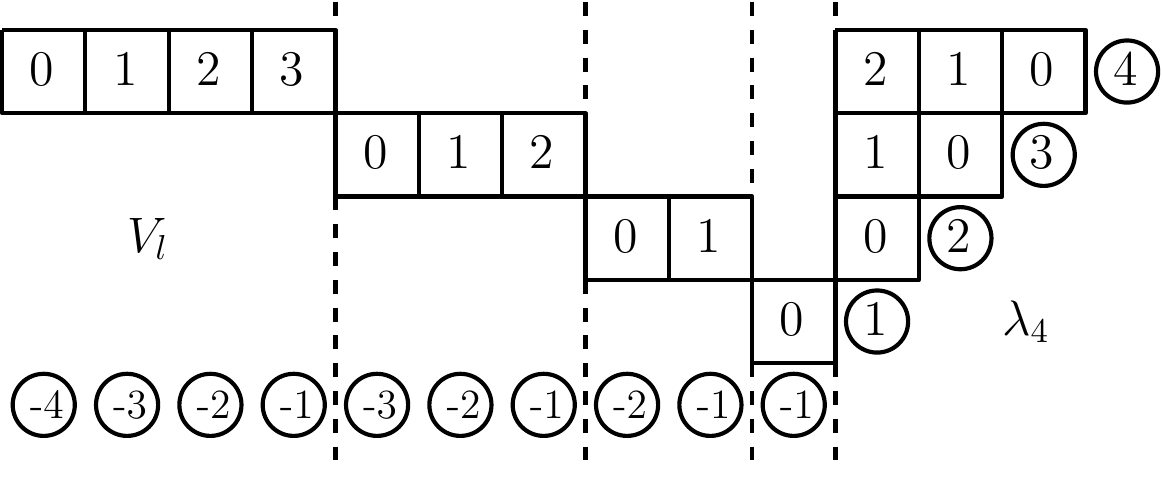}
\caption{}\label{FigFlag3}
\end{figure}
\end{example}

Our aim is now to show that $(IS(\G(r,n)),\subseteq)$ and $(G(r,n),\leq_f)$ are isomorphic, by constructing an explicit poset isomorphism. For that purpose, we will split our study into two distinct steps: we will first construct a bijection between $IS(\G(r,n))$ and $G(r,n)$ (see Definition~\ref{DefFlagBij} and Proposition~\ref{Flag1}), and then we will show that this bijection is in fact a poset isomorphism (see Theorem~\ref{Flag2}). We begin with a lemma, which shows how we can associate a permutation with each element of $IS(\G(r,n))$.

\begin{lemma}\label{FlagSym}
Let $U \in IS(\G(r,n)$, then $U \cap V_{A,n}$ is the inversion set of a permutation in $A_{n-1}$.
\end{lemma}

\begin{proof}
Let us denote by $X$ the set $U \cap V_{A,n}$ and by $E'$ the set of arcs of $\A$, where $\A$ is the valued digraph associated to $(S_n,\leq_R)$ defined in Section~\ref{SectSym}.
Since $\lambda_n=V_{A,n}$, if $X \in IS(\A)$, then $X$ is the inversion set of a permutation thanks to Corollary~\ref{CorolFinalA}. We still have to show that $X$ is in $IS(\A)$. First, notice that for all $\mathfrak{c}$ in $V_{A,n}$, we have $H(\mathfrak{c})=H_f(\mathfrak{c})$,
where $H(\mathfrak{c})$ is the hook based on $\mathfrak{c}$ in $\lambda_n$ defined in Section~\ref{SectSym}. Thus, by definition of the underlying digraph of $\G(r,n)$, for all $\mathfrak{c} \in V_{A,n}$ we have
\begin{align*}
\{ \ \mathfrak{d} \in \lambda_n \ | \ \mathfrak{d}\in U , \ (\mathfrak{c},\mathfrak{d}) \in E \}&=U \cap H_f(\mathfrak{c}) \\
&=U \cap H(\mathfrak{c}) \\
&=\{ \ \mathfrak{d}\in \lambda_n \ | \ \mathfrak{d}\in U, \ (\mathfrak{c},\mathfrak{d}) \in E' \}.
\end{align*}
Consequently, for all $\mathfrak{c} \in \lambda_n$, if $\mathfrak{c} \in U$, then by Proposition~\ref{Technic} we have
\begin{align*}
\theta(\mathfrak{c}) &\leq |\{  \mathfrak{d} \in \lambda_n \ | \ \mathfrak{d}\in U, \ (\mathfrak{c},\mathfrak{d}) \in E \}|\\
&\leq |\{  \mathfrak{d}\in \lambda_n \ | \ \mathfrak{d}\in U, \ (\mathfrak{c},\mathfrak{d}) \in E' \}|,
\end{align*}
and the converse inequality holds when $\mathfrak{c} \notin U$. Thus, $X$ is in $IS(\A)$ by Proposition~\ref{Technic}, and this ends the proof.
\end{proof}

Thanks to Lemma~\ref{FlagSym}, one can associate to each element of $IS(\G(r,n))$ a permutation in $S_n$. What remains to understand is how to associate a color to each value of the permutation. For that purpose, we introduce a new notation.

\begin{definition}
Let $U \in IS(\G(r,n))$ and $i \in [n]$, we define the following two quantities
\begin{equation*}
R_i(U) := |\{  (x,i) \in U \ | \ (x,i) \in V_{A,n} \}| \ \text{and} \
L_i(U) := |\{  (x,i) \in U \ | \ (x,i) \in V_{B,r,n} \}|.
\end{equation*}
\end{definition}

\begin{lemma}\label{FlagColor}
Let $U \in IS(\G(r,n))$ and $i \in [n]$. Then, we have
\[0 \leq L_i(U) - (r-1)R_i(U) \leq r-1.\]
\end{lemma}

\begin{proof}
By definition, for all $(x,i) \in V_{B,r,n}$ and $(y,i) \in V_{A,n}$, we have $((x,i),(y,i))\in E$. Thus, we have for all $(x,i) \in V_{B,r,n}$
\begin{equation*}
 R_i(U) \leq | \{ \mathfrak{d} \in U \ | \ ((x,i),\mathfrak{d}) \in E \} |.
\end{equation*}
Let us consider the following set
\[ X= \{(x,i) \in V_{B,r,n} \ | \ -i(r-1)\leq x < -(i-R_i(U))(r-1) \}. \] 
Clearly, we have $|X| = (r-1)R_i(U)$. Moreover, by definition we have for all $(x,i) \in X$ that
\begin{align*}
&\frac{x}{r-1}<R_i(U)-i \ \
\Longrightarrow  \ \ \left\lfloor\frac{x}{r-1} \right\rfloor<R_i(U)-i \ \
\Longrightarrow  \ \ i+\left\lfloor\frac{x}{r-1} \right\rfloor<R_i(U) \\
\Longrightarrow  \ \ &\theta(x,i) < R_i(U) \ \
\Longrightarrow  \ \ \theta(x,i) < | \{ \mathfrak{d} \in U \ | \ ((x,i),\mathfrak{d}) \in E \} |.
\end{align*}
We thus have $X \subseteq U$ by Proposition~\ref{Technic}. Therefore, we have
\begin{equation}\label{EqProofFlag1}
(r-1)R_i(U)=|X| \leq L_i(U) \ \ \Longrightarrow \ \ 0 \leq L_i(U) - (r-1)R_i(U). 
\end{equation}
To prove the converse inequality, we consider the set
\[ Y= \{(y,i) \in V_{B,r,n} \ | \ -(i-R_i(U)-1)(r-1)\leq y \leq -1 \}. \] 
For all $(y,i) \in Y$, we have
\begin{equation*}
R_i(U)-i+1 \leq  \frac{y}{r-1} \ \
\Longrightarrow  \ \ R_i(U)-i <  \left\lfloor\frac{y}{r-1}\right\rfloor \ \
\Longrightarrow  \ \ R_i(U) <  \theta(y,i).
\end{equation*}
Let us now fix $(y,i)$ in $Y$. We will show by backward induction on $y$ that $(y,i) \notin U$. By definition, each arc having $(-1,i)$ as starting point has an element of $V_{A,n}$ as ending point. Moreover, such an arc has its ending point in row $i$, so that we have
\[ | \{ \mathfrak{d} \in U \ | \ ((-1,i),\mathfrak{d}) \in E \} | \ \leq \ R_i(U) \ < \ \theta(-1,i).\]
Thus, $(-1,i) \notin U$ by Proposition~\ref{Technic}, and one can finish the induction using similar arguments. Consequently, $Y \cap U$ is empty. However, we have $|Y|=(i-R_i(U)-1)(r-1)$, so that 
\begin{equation}\label{EqFlagProof2}
L_i(U) \leq i(r-1)-|Y| \ = \ (R_i(U)+1)(r-1) \ \Longrightarrow \ L_i(U)-R_i(U)(r-1) \leq r-1.
\end{equation}
Combining \ref{EqFlagProof2} and \ref{EqProofFlag1}, we have the expected result.
\end{proof}

Thanks to Lemma~\ref{FlagColor} and Lemma~\ref{FlagSym}, we are now able to associate a $r$-colored permutation to each element of $IS(\G(r,n))$.

\begin{definition}\label{DefFlagBij}
Let $U \in IS(\G(r,n))$. 
We denote by $\sigma_U$ the unique permutation such that $\Inv(\sigma_U)=U \cap V_{A,n}$, and we denote by $(c_i(U))_{1 \leq i \leq n}$ the sequence defined by 
\[ c_{\sigma^{-1}(i)}(U):=L_i(U)-(r-1)R_i(U).\]
We denote by $\Psi$ the map from $IS(\G(r,n))$ to $G(r,n)$ defined by
\[ \text{for all} \ U\in IS(\G(r,n)), \ \Psi(U) \ = \ ((c_i(U))_{1 \leq i \leq n},\sigma_U). \]
\end{definition}

In what follows, we will show that the function $\Psi$ is a poset isomorphism between $(IS(\G(r,n)),\subseteq)$ and $(G(r,n),\leq_f)$. We first give a technical lemma, which is useful for both step of our proof.

\begin{lemma}\label{FlagLeft}
Let $U \in IS(\G(r,n))$ and $(a,b) \in V_{B,r,n}$. Then,
$ \text{for all} \ k\in \mathbb{Z} \ \text{such that}$
 \[-b(r-1)\leq k < a,\]
$\text{if} \ (a,b) \in U, \ \text{then} \ (k,b) \in U.$ 
On the representation of $\G(r,n)$ as a diagram, this means that if a box of $V_{B,r,n}$ is in $U$, then all the boxes which are strictly on its left and in the same row are also in $U$.
\end{lemma}

\begin{proof}
Let $k$ in $\mathbb{Z}$ be such that $-b(r-1)\leq k < a$ and assume that $(a,b) \in U$. By Proposition~\ref{Technic}, we have
\[ \theta(a,b) \leq |\{\mathfrak{d} \in U \ | \ ((a,b),\mathfrak{d})\in E \}|. \]
Let us consider $\mathfrak{d}\in U$ such that $((a,b),\mathfrak{d})\in E$. By construction of the underlying digraph of $\G(r,n)$, we have $((k,b),\mathfrak{d}) \in E$. Moreover, $((k,b),(a,b))$ is also in $E$, so that
\[ |\{\mathfrak{d} \in U \ | \ ((a,b),\mathfrak{d})\in E \}|<|\{\mathfrak{d} \in U \ | \ ((k,b),\mathfrak{d})\in E \}|. \]
Finally, by definition of $\theta$ we have $\theta(k,b) \leq \theta(a,b)$, hence we have
\[ \theta(k,b) < |\{\mathfrak{d} \in U \ | \ ((k,b),\mathfrak{d})\in E \}|. \]
Thus, $(k,b) \in U$ by Proposition~\ref{Technic}, and this concludes the proof.
\end{proof}

\begin{proposition}\label{Flag1}
The function $\Psi$ is a bijection.
\end{proposition}

\begin{proof}
Let us first prove that $\Psi$ is injective. Let $U,U'\in \G(r,n)$ such that $\Psi(U)=\Psi(U')$. Then $\sigma_U=\sigma_{U'}$, hence $\Inv(\sigma_U)=\Inv(\sigma_{U'})$, so that $U \cap V_{A,n}=U'\cap V_{A,n}$.
Therefore, we have $R_j(U)=R_j(U')$ for all $j \in [n]$. 
Let us now fix $j \in [n]$, by definition of $(c_i)_{1 \leq i\leq n}$ we have 
\[ L_j(U)-(r-1)R_j(U)=L_j(U')-(r-1)R_j(U'), \]
so that $L_j(U)=L_j(U')$. Thus, the number of boxes that are in $V_{B,r,n} \cap U$ and in row $j$ equals the number of the boxes that are in $V_{B,r,n} \cap U'$ and in row $j$. However, thanks to Lemma~\ref{FlagLeft}, these boxes are left-justified in $V_{B,r,n}$, hence we have
\[ \{(a,j) \in V_{B,r,n} \ | \ (a,j) \in U \}=\{(a,j) \in V_{B,r,n} \ | \ (a,j) \in U' \}. \]
Thus, $U \cap V_{B,r,n}=U' \cap V_{B,r,n}$ so $U=U'$, and this proves that $\Psi$ is injective.

We now prove that $\Psi$ is surjective. Let $\pi=((c_i)_i,\sigma) \in G(r,n)$, we denote by $R_i$ the quantity defined by
\[ R_i:=|\{(x,i) \in V_{A,n} \ | \ (x,i) \in \Inv(\sigma) \}|, \]
and we define the following sets:
\[ \text{for all} \ i\in[n], \ U_i=\{(x,i) \ | \ -i(r-1) \leq x < -(i-R_i)(r-1)+c_{\sigma^{-1}(i)} \}. \]
We will prove that $U:=\Inv(\sigma)\cup U_1 \cup U_2 \cup \ldots \cup U_n$ is in $IS(\G(r,n))$. For that purpose, let us consider $\mathfrak{c} \in \lambda_{r,n}$ and divide our study into four cases.
\begin{enumerate}
\item (case $\mathfrak{c}\in U \cap V_{A,n}$). Following a similar method as in the proof of Lemma~\ref{FlagSym}, one can show that 
\[
\theta(\mathfrak{c}) \leq |\{  \mathfrak{d} \in U \ | \ (\mathfrak{c},\mathfrak{d}) \in E \}|
\]
\item (case $\mathfrak{c}\in V_{A,n} \setminus U$). A similar argument as in Case (1) shows that 
\[
\theta(\mathfrak{c}) \geq |\{  \mathfrak{d} \in U \ | \ (\mathfrak{c},\mathfrak{d}) \in E \}|
\]
\item (case $\mathfrak{c} \in U \cap V_{B,r,n}$). We set $(x,i)=\mathfrak{c}$. By definition of $U_i$, we have
\begin{align*}
& \ x < -(i-R_i)(r-1)+c_{\sigma^{-1}}(i) \
\Longrightarrow \ x < -(i-R_i)(r-1)+(r-1) \\
&\Longrightarrow \ \frac{x}{r-1} < 1+R_i-i \
\Longrightarrow \ \left\lfloor\frac{x}{r-1}\right\rfloor \leq R_i-i \
\Longrightarrow \ \theta(\mathfrak{c}) \leq R_i,
\end{align*}
but $R_i \leq |\{  \mathfrak{d} \in U \ | \ (\mathfrak{c},\mathfrak{d}) \in E \}|$ by definition of the digraph, hence we have 
\[ \theta(\mathfrak{c}) \leq |\{  \mathfrak{d} \in U \ | \ (\mathfrak{c},\mathfrak{d}) \in E \}|. \]
\item (case $\mathfrak{c} \in V_{B,r,n} \setminus U$). A similar argument as in case (3) shows that 
\[ \theta(\mathfrak{c}) \geq |\{  \mathfrak{d} \in U \ | \ (\mathfrak{c},\mathfrak{d}) \in E \}|. \]
\end{enumerate}
Consequently, $U \in IS(\G(r,n))$ thanks to Proposition~\ref{Technic}, and by construction we have $\Psi(U)=\pi$. Thus, $\Psi$ is surjective, and this conclude the proof.
\end{proof}

We now prove that $\Psi$ is a morphism of posets.

\begin{proposition}\label{Flag2}
Let $U,U' \in IS(\G(r,n))$, we have that $U$ covers $U'$ in $(IS(\G(r,n)),\subseteq)$ if and only if $\Psi(U)$ covers $\Psi(U')$ in the flag weak order.
\end{proposition}

\begin{proof}
We set
\[ \Psi(U):=((c_i)_i,\sigma)=\pi \ \text{and} \ \Psi(U'):=((c'_i)_i,\omega)=\pi'. \]

Assume that $U$ covers $U'$ in $(IS(\G(r,n)),\subseteq)$.
Since $(IS(\G(r,n)),\subseteq)$ is graded, there exists $(x,y) \in \lambda_{r,n} \setminus U$ such that 
\[ U=U' \cup \{ (x,y) \}. \]
We will prove that $\Psi(U)$ covers $\Psi(U')$ using Lemma~\ref{CaracCover}. There are two cases.
\begin{itemize}
\item (Case $(x,y) \in V_{A,n}$). We have that $\omega$ is obtained from $\sigma$ by swapping positions of $x$ and $y$. Moreover, by definition of $\Psi$ we have the following two facts:
\[ \sigma^{-1}(x)=\sigma^{-1}(y)+1 \ \text{and} \ (x,y) \notin \Inv(\sigma), \]
so that $x$ and $y$ are adjacent in $\sigma$. It remains to show that $c_{\sigma^{-1}(y)}(U)=r-1$. For the sake of clarity, let us denote by $i$ the integer $\sigma^{-1}(y)$, and assume by contradiction that $c_{i}(U)<r-1$. Since we have 
\[ c_i(U)=L_i(U)-(r-1)R_i(U), \L_i(U)=L_i(U') \ \text{and} \ R_i(U')=R_i(U)+1, \]
we thus have $L_i(U')-(r-1)R_i(U')<0$, and this contradicts Lemma~\ref{FlagColor}. Therefore, we have 
$\pi'=\pi.a_i, \ c_{i+1}=r-1 \ \text{and} \ \sigma(i)<\sigma(i+1)$,
so that $\pi'$ covers $\pi$ in $G(r,n)$ by Lemma~\ref{CaracCover}.
\item (Case $(x,y) \in V_{B,r,n}$). As in the previous case, let us denote by $i$ the integer $\sigma^{-1}(y)$, and assume by contradiction that $c_{i}(U)=r-1$. By definition, we have 
\[  c_i(U)=L_i(U)-(r-1)R_i(U)=r-1 \ \text{and} \ L_i(U')=L_i(U)+1, \] 
so that $c_i(U')=r$, which is absurd. Thus, we have $c_i(U) < r-1$ and it is clear that $\pi'=\pi.b_i$. Consequently, $\pi'$ covers $\pi$ by Lemma~\ref{CaracCover}.
\end{itemize}

We now prove the converse. 
If $\pi'=\pi. b_i$, then $c_i(B)=c_i(A)+1$, so that $U'$ is obtained from $U$ by adding just one box in the $i$-th line of $V_{B,r,n}$ by Lemma~\ref{FlagLeft}. Thus $U'$ covers $U$. If $\pi'=\pi. a_i$, then a straightforward calculation using the definition of the function $\Psi$ shows that $U'=U \cup \{(\sigma^{-1}(i),\sigma^{-1}(i+1)\}$, so that $U'$ covers $U$. This concludes the proof.
\end{proof}

As an immediate consequence of Propositions \ref{Flag1} and \ref{Flag2}, we have the following corollary, which concludes this section.

\begin{corollary}
The posets $(G(r,n),\leq_f)$ and $(IS(\G(r,n)),\subseteq)$ are isomorphic.
\end{corollary}

\subsection{Down-set (resp. up-set) lattice of a finite poset}\label{SectionDownSet}

In this section, we consider $\Pp=(P,\leq)$ a finite poset. Let us denote by $G=(V,E)$ the digraph defined by
\[V:=P \ \text{and} \ E:=\{(x,y)\in P^2 \ | \ x\neq y \ \text{and} \ x \leq y \}. \]
It is clear that $G$ is a simple acyclic digraph, and we denote by $\G(P)=(G,\theta)$ the valued digraph such that for all $z \in P$, $\theta(z)=0$. 

\begin{proposition}\label{PropLinearExtens}
The set $PS(\G(P))$ equals the set of the linear extensions of $P$.
\end{proposition}

\begin{proof}
Let us perform the peeling process on $\G(P)$. By definition of $\theta$, a vertex $z \in P$ is erasable in $\G(P)$ if and only if we have the following property: $\text{for all} \ y \in P$
\[ \text{if} \ y \leq z, \ \text{then} \ y=z, \]
\emph{i.e.} $z$ is a minimum of $(P,\leq)$. Let us denote by $\G(P)'$ the valued digraph obtained after we peeled a minimum element $z$ of $P$, and denote by $P'$ the poset $(P \setminus \{z\},\leq)$. Clearly, we have
\[ \G(P)'=\G(P'). \]
Therefore, applying the peeling process on $\G(P)$ is equivalent to performing on $P$ the process described in the introduction of Section~\ref{SecDef}. It follows that each peeling sequence of $\G(P)$ is a linear extension of $P$.
The converse implication can be easily proved by induction on the cardinality of $P$.
\end{proof}

We have the following immediate corollary.

\begin{corollary}
The poset $(IS(\G(P)),\subseteq)$ is isomorphic to the down-set lattice of $(P,\leq)$.
\end{corollary}

\begin{Remark}
Note that one can obtain the up-set lattice of $(P,\leq)$ by the same method, considering the same digraph $G$ endowed with the valuation $\eta$ defined by
\[ \text{for all} \ z\in P, \ \eta(z) := d^+(z). \]
\end{Remark}

\section{Generalized columns and Quasi-symmetric functions}\label{SecQuasi}

Symmetric functions can be defined as the homogeneous formal power series $F(x_1,x_2,\ldots)$ in countable infinitely many variables being invariant under the action of the symmetric group. That is, for any monomial $X=x_{i_1}\cdots x_{i_k}$ appearing in $F$ and for any simple transposition $s_j \in S_n$, the monomials $X$ and $s_j.X$ have the same coefficient, where $s_j.X$ denote the monomial obtained by permuting the variables $x_j$ and $x_{j+1}$ in $X$. 
Quasi-symmetric functions admits a similar definition. We say that $F(x_1,x_2,\ldots)$ is a \emph{quasi-symmetric function} if and only if for any monomial $X=x_{i_1}\cdots x_{i_k}$ appearing in $F$ and for any simple transposition $s_j$ such that not both $x_j$ and $x_{j+1}$ appear in $X$, then $X$ and $s_j.X$ have the same coefficient in $F$. In particular, symmetric functions are quasi-symmetric functions. 

A useful basis of the space of quasi-symmetric 
functions (of degree $n$) is given by the \emph{fundamental quasi-symmetric functions}, introduced by Gessel in \cite{G84} (see also \cite{S2}~(7.81)).
They are defined as follows: for any $X \subseteq [n-1]$ we set
\[
G\! _X^{ \ \! n}(x_1,x_2,...):= \sum_{\substack{i_1\leq i_2 \leq \ldots \leq i_{n} \\ i_j<i_{j+1} \ {\rm if} \ j\in X}}x_{i_1}\cdots x_{i_n},
\]
that we will generally denote $G\! _X$ when there is no ambiguity.

\subsection{Linear extensions and quasi-symmetric functions}

The first occurrence of quasi-symmetric function goes back to the thesis work of Stanley, via the notion of $P$-partition which generalizes the concept of classical partition of an integer. A $P$-partition is the couple of a finite poset $P$ (with $|P|=n$), together with a given bijection $\gamma$ from $P$ to $[n]$. For any linear extension $L=[z_1,\ldots,z_n]$ of $P$, let ${\rm Des}(L,\gamma)$ be the set of all the indices $j\in [n-1]$ such that $\gamma(j)>\gamma(j+1)$, called the \emph{descent set} of $L$. Stanley associates in \cite{S3} a formal power series with the $P$-partition $(P,\gamma)$ as follows:
\begin{equation}\label{DefGamma}
\Gamma(P,\gamma):= \sum_{L} G_{{\rm Des}(L,\gamma)},
\end{equation}
where the sum is over all linear extensions of $P$. Note that this is a classical reformulation of the original definition, that we give in the following proposition.

\begin{proposition}\label{PropPPart}
A $(P,\gamma)$-partition is a function $f$ from $P$ to $\mathbb{N}^*$ such that there exists a linear extension $L=[z_1,\ldots,z_n]$ of $P$ which satisfies:
\begin{enumerate}
\item for all $1\leq i <j \leq n$, $f(z_i) \leq f(z_j)$;
\item for all $1\leq i<j\leq n$, if $\gamma(z_i) > \gamma(z_j)$, then $f(z_i)<f(z_j)$.
\end{enumerate}
We have $\displaystyle{\Gamma(P,\gamma)= \sum_{f}\prod_{p \in P}x_{f(p)}}$, where the sum is over all $(P,\gamma)$-partitions.
\end{proposition}

\subsection{Definition of the formal power series}

Thanks to Section~\ref{SectionDownSet}, we have that the notion of peeling sequence is a generalization of linear extension of a finite poset to a valued digraph. Thus, it is natural to look for a generalization of \eqref{DefGamma} to the case of valued digraphs. This is the point of this section. We begin with introducing a useful notation.

\begin{definition}
Let $\G=(G,\theta)$ be a valued digraph. For all $A \in (IS(\G),\subseteq)$ we denote by $PS(A)$ the set defined by 
\begin{equation*}
PS_A(\G):=\left\{[z_1,\ldots,z_{|A|}] \ 
\left| 
\begin{aligned}
& A=\{z_1,\ldots,z_{|A|}\} \\
& \exists [x_1,x_2,\ldots]\in PS(\G) \ \text{such that} \ x_i=z_i \ \text{for all} \ i\leq |A|
\end{aligned}
\right.
\right\}.
\end{equation*}
\end{definition}

A straightforward way to generalize the series $\Gamma(P,\gamma)$ would be to consider a bijection $\mu$ from the vertices of $\G$ to $\{1,\ldots,|V|\}$, and directly adapt \eqref{DefGamma} to this new context. However, in the sequel we will need a slightly more general definition, which is inspired by the \emph{column-strictness conditions} introduced in \cite{FGR} and \cite{YY}.

\begin{definition}
A set of \emph{generalized columns} of $\G$ is a family $\U=(\U_z)_{z\in V}$ of subsets of $V$. Let $A \in IS(\G)$, $\U$ be a set of generalized columns and $f$ be a function from $A$ to $\mathbb{N}$. We say that $f$ is a \emph{$(A,\U)$-semi-standard function} if and only if there exists $L=[z_1,\ldots,z_n]\in PS_A(\G)$ such that:
\begin{enumerate}
\item for all $1\leq i <j \leq n$, we have $f(z_i) \leq f(z_j)$;
\item for all $1\leq i<j\leq n$, if $z_j \in \U_{z_i}$, then $f(z_i)<f(z_j)$.
\end{enumerate}
Such a peeling sequence is called a \emph{$f$-compatible peeling sequence} of $A$.
We denote by $\SSF(A,\U)$ the set of all the $(A,\U)$-semi-standard functions (when there is no ambiguity, we will simply denote it by $\SSF(A)$). Finally, we define the formal power series
\[
\Gamma(A,\U):=\sum_{f \in \SSF(A)}\prod_{z \in A} x_{f(z)}.
\]
\end{definition}

\begin{proposition}
The series $\displaystyle{\Gamma(A,\U)}$ is a quasi-symmetric function.
\end{proposition}

\begin{proof}
Let $n=|A|$, $i \in \mathbb{N}^*$ and $f \in \SSF(A)$ such that $f^{-1}(\{ i \}) \neq \emptyset$ and $f^{-1}(\{ i+1 \})= \emptyset$. Let $L=[z_1,\ldots,z_n]$ be a $f$-compatible peeling sequence of $A$ and denote by $\widehat{f}$ the function from $A$ to $\mathbb{N}^*$ defined by 
\begin{equation*}
\text{for all} \ z \in A, \ \widehat{f}(z)= \left\{
\begin{aligned}
&i+1 \ \text{if} \ z \in f^{-1}(\{ i \}), \\
&f(z) \ \text{otherwise}.
\end{aligned}
\right.
\end{equation*}
We will prove that $L$ is a $\widehat{f}$-compatible sequence.
Since $f \in \SSF(A)$, then $\widehat{f}$ is weakly increasing along $[z_1,\ldots,z_n]$. Assume by contradiction that there exists $p<q$ such that $z_q \in \U_{z_p}$ and $\widehat{f}(z_p)=\widehat{f}(z_q)$.
Then, we have $f(z_p)=f(z_q)$ and this contradicts the fact that $L$ is $f$-compatible. Therefore, $L$ is $\widehat{f}$-compatible, hence $\widehat{f}$ is an $(A,\U)$-semi-standard function, and this is enough to prove that $\displaystyle{\Gamma(A,\U)}$ is quasi-symmetric. This concludes the proof.
\end{proof}

It appears that $\Gamma(A,\U)$ is a generalization of the function associated with a $P$ partition, thanks to the following immediate proposition.

\begin{proposition}
Let $(P,\leq)$ be a finite poset, $(P,\gamma)$ be a $P$-partition and $\G(P)$ be the valued digraph defined in Section~\ref{SectionDownSet}. We set $\U_z(\gamma):=\{y \in P \ | \ \gamma(z)>\gamma(y) \}$ and $\U(\gamma)=(\U_z(\gamma))_{z\in P}$ a set of generalized columns of $\G(P)$. Then
\[ 
\Gamma(P,\U(\gamma))=\Gamma(P,\gamma).
\]
\end{proposition}

\begin{Question}
As suggested before, we could have defined $\Gamma(A,\U)$ associating a descent set to each element of $PS_A(\G)$ in the obvious way and then summing all the associated fundamental quasi-symmetric functions. However, it seems not to be any particular reason to expect these two definitions to coincide in general. It should be interesting to investigate if there exists some valued digraphs with a choice of generalized columns for which this equality occurs.
\end{Question}

We finish this section with an obvious lemma connecting this function to the enumeration of maximal chains in $(IS(\G),\subseteq)$.

\begin{lemma}
Let $\G$ be a valued digraph, $\U$ be a set of generalized columns of $\G$ and $A$ be an element of $IS(\G)$. Then, the coefficient of $x_1x_2\cdots x_n$ in $\Gamma(A,\U)$ is equal to the number of maximal chain from $\emptyset$ to $A$ in $(IS(\G),\subseteq)$.
\end{lemma}

\begin{proof}
This is clear by definition of $\SSF(A)$.
\end{proof}

\subsection{Type $A$ and Stanley's symmetric function}

In this section, we consider the valued digraph $\A=(G,\theta)$ associated with the weak order on $A_{n-1}$ (see Section~\ref{SectSym}). Since $\A$ can be seen as the Ferrers diagram of the partition $\lambda_n$, we have a natural choice for a set of generalized columns, given by the columns of $\A$.

\begin{definition}
The set of generalized columns $\U^{col}=(\U_{(a,b)})_{1\leq a<b \leq n}$ of $\A$ is defined by: 
\[
 \U_{(a,b)}:=\{(a,k) \ | \ a< k \leq n \}.
\]
\end{definition}

Surprisingly, the series which arise from this choice of generalized columns are the well-known Stanley symmetric functions of type $A$ (see \cite{S1}). Let us first recall the definition of Stanley symmetric functions.

\begin{definition}
Let $\sigma \in S_n$, the \emph{Stanley symmetric function} associated with $\sigma$ is the formal power series defined by:
$$F_{\sigma}(x_1,x_2,\ldots):=\sum_{(i_1,\ldots,i_{\ell(\sigma)}) \in {\rm Red}(\sigma)} \ \sum_{\substack{r_1\leq r_2 \leq \ldots \leq r_{\ell(\sigma)} \\ r_j<r_{j+1} \ {\rm if} \ i_j<i_{j+1}}} \ x_{r_1}x_{r_2}\cdots x_{r_{\ell(\sigma)}},$$ where ${\rm Red}(\sigma)$ denote the set of the reduced decompositions of $\sigma$.
\end{definition} 

In \cite{FGR}, the authors give a characterization of such function indexed by a permutation $\sigma \in S_n$, in terms of sums over a set of tableaux called \emph{balanced labellings} of the \emph{Rothe diagram} of $\sigma$. We will prove that the series arising from our description are exactly the Stanley symmetric functions. we follow the same method as them. Note that another method consists in constructing an explicit bijection between these balanced labellings and the elements of $\SSF(\Inv(\sigma),\U^{col})$, but we will not detail this here.

\begin{Theorem}\label{TheoStanSymMoi}
For all $\sigma \in S_n$, we have: \[\Gamma(\Inv(\sigma),\U^{col})=F_{\sigma}.\]
\end{Theorem}

Before giving the proof, we need the following technical lemma.

\begin{definition}
Let $\sigma \in S_n$, $f \in \SSF(\Inv(\sigma),\U^{col})$ and $M={\rm max}\{f(c) \ | \ c\in \Inv(\sigma) \}$. The \emph{leading cell} of $f$ is the unique element $(a,b) \in \Inv(\sigma)$ such that:
\begin{enumerate}
\item $f(a,b)=M$;
\item the integer $\sigma^{-1}(a)$ is minimal such that (1) is true.
\end{enumerate} 
\end{definition}

\begin{lemma}\label{TechStanSym}
Let $\sigma \in S_n$, $f\in \SSF(\Inv(\sigma))$ and $(a,b)$ be the leading cell of $f$. Then, there exists $\omega \in S_n$ such that $\Inv(\omega)=\Inv(\sigma)\setminus \{(a,b)\}$.
\end{lemma}

\begin{proof}
Let $M=f(a,b)$ and $L=[(a_1,b_1),\ldots,(a_{\ell(\sigma)},b_{\ell(\sigma)})]\in PS_{\Inv(\sigma)}(\A)$ be a $f$-compatible peeling sequence. 
Thanks to Corollary~\ref{CorolFinalA}, there exists $\sigma_1,\sigma_2,\ldots,\sigma_{\ell(\sigma)} \in S_n$ such that:
\begin{enumerate}
\item $Id \lhd_R \sigma_1 \lhd_R \ldots \lhd_R \sigma_{\ell(\sigma)}=\sigma$;
\item $\Inv(\sigma_i)=\{(a_1,b_1),\ldots,(a_i,b_i) \}$.
\end{enumerate}
There exists $k$ such that $(a_k,b_k)=(a,b)$, hence $a$ and $b$ are adjacent in $\sigma_{k-1}$ and $\sigma_k$ is obtained from $\sigma_{k-1}$ by swapping the positions of $a$ and $b$. Let us assume that the position of $a$ and $b$ in $\sigma_k$ is preserved in $\sigma$, that is we have
\begin{equation}\label{EqTechProofQuasiSym1}
\sigma_{k}=[\sigma_k(1),\ldots,b,a,\ldots,\sigma_k(n)] \ \ \ \text{and} \ \ \ \sigma=[\sigma(1),\ldots,b,a,\ldots,\sigma(n)].
\end{equation}
In that case, the permutation $\omega$ obtained from $\sigma$ by swapping the positions of $b$ and $a$ satisfy $\Inv(\omega)=\Inv(\sigma)\setminus \{(a,b)\}$, which is exactly the expected result. We still have to prove that \eqref{EqTechProofQuasiSym1} is true. 
For that purpose, we will show that $a_p \neq a$ and $b_p \neq b$ for all $p>k$, implying that the positions of $a$ an $b$ stay the same in $\sigma_k, \sigma_{k+1},\ldots,\sigma_{\ell(\sigma)}=\sigma$.

 Assume by contradiction that there exists $p>k$ such that $a_p=a$. Then, we have $f(a,b_p)=M=f(a,b)$ because $f$ is weakly increasing along $L$. But $(a,b_p)$ and $(a,b)$ are in the same column of $\lambda_n$, and this is absurd by definition of $\U^{col}$. Similarly, assume by contradiction that there exists $p>k$ such that $b_p=b$. Then, for all $q>p$ we have $(a_q,b_q) \neq (a_p,a)$ (otherwise $(a_p,b_p)$ and $(a_q,b_q)$ would be both in the column $a_p$, but it is impossible since $f$ takes the value $M$ on both of them), hence we have
\[ \sigma=[\sigma(1),\ldots,a_p,\ldots,a,\ldots,\sigma(n)]. \]
Thus, we have $\sigma^{-1}(a_p)<\sigma^{-1}(a)$, and this contradicts the minimality of $\sigma^{-1}(a)$. 

Consequently, \ref{EqTechProofQuasiSym1} is true, and this concludes the proof.
\end{proof}

We now have everything we need to prove Theorem~\ref{TheoStanSymMoi}.

\begin{proof}[Proof of Theorem~\ref{TheoStanSymMoi}]
We will define a bijection $\Psi$ associating a $(\Inv(\sigma),\U^{col})$-semi-standard function $f$ with a pair of sequences, denoted by
\begin{equation}\label{DefBijStanTypA}
\Psi(f)=( \ [i_1,\ldots,i_k] \ , \ [r_1,\ldots,r_k] \ ),
\end{equation}
such that 
\begin{enumerate}
\item $s_{i_1}\cdots s_{i_k}$ is a reduced decomposition of $\sigma$;
\item $(r_j)_j$ is weakly increasing and $r_{j}<r_{j+1}$ whenever $i_j<i_{j+1}$;
\item $\displaystyle{\prod_{1 \leq j \leq k}x_{r_j}=\prod_{z \in \Inv(\sigma)}f(z)}$.
\end{enumerate} 
Clearly, if such a bijection exists, then \eqref{EqObjectifSym} is true.
We split our proof into two step: first, we will define a function satisfying all the required conditions; then, we will prove that it is bijection, constructing its reverse function. 

\textbf{Step 1: Definition of $\Psi$.} Let $f \in \SSF(\Inv(\sigma))$, we define by backward induction a pair of sequences $[i_1,\ldots,i_{\ell(\sigma)}]$ and $[r_1 \leq \ldots \leq r_{\ell(\sigma)}]$ using Lemma~\ref{TechStanSym} as follows:
\begin{itemize}
\item let $(a,b) \in \Inv(\sigma)$ be the leading cell of $f$;
\item let $\omega \in S_n$ be such that $\Inv(\omega)=\Inv(\sigma) \setminus \{(a,b)\}$ and set $i_{\ell(\sigma)}$ and $r_{\ell(\sigma)}$ the two integers such that $\sigma=\omega s_{i_{\ell(\sigma)}}$ and $r_{\ell(\sigma)}=f(a,b)$;
\item repeat this procedure swapping $\sigma$ with $\omega$ and $f$ with $g:=f|_{\Inv(\omega)}$, and so on.
\end{itemize}
Let us now check that this pair of sequences satisfies conditions (1), (2) and (3). Clearly, we have $r_1 \leq r_2 \leq \ldots \leq r_{\ell(\sigma)}$, and thanks to Lemma~\ref{TechStanSym} $s_{i_1}\cdots s_{i_{\ell(\sigma)}}$ is in ${\rm Red}(\sigma)$. We still have to prove that $r_j<r_{j+1}$ whenever $i_j<i_{j+1}$. We will prove the contrapositive: let $j$ be such that $r_j=r_{j+1}$, and denote by $a$ and $b$ the two integers such that 
\begin{align*}
s_{i_1}\cdots s_{i_{j-1}}&=[\ldots,a,b,\ldots], \\
s_{i_1}\cdots s_{i_{j}}&=[\ldots,b,a,\ldots].
\end{align*} 
Our aim is now to prove that $i_{j+1}$ is strictly smaller than $i_j$. Assume by contradiction that $i_j \leq i_{j+1}$, and consider the three following cases.
\begin{itemize}
\item If $i_j=i_{j+1}$, then $s_{i_1}\cdots s_{i_{\ell(\sigma)}}$ is not reduced, which is absurd.
\item If $i_{j+1}=i_j+1$, then there exists $c$ such that we have 
\begin{align*}
s_{i_1}\cdots s_{i_{j}}&=[\ldots,b,a,c,\ldots], \\
s_{i_1}\cdots s_{i_{j+1}}&=[\ldots,b,c,a,\ldots].
\end{align*} 
However, $(a,b)$ and $(a,c)$ are in the same column of $\lambda_n$, and by hypothesis we have $f(a,b)=r_j=r_{j+1}=f(a,c)$, and this contradicts the fact that $f \in \SSF(\Inv(\sigma),\U^{col})$.
\item If $i_{j+1}>i_{j}+1$, then there exists two integers $c$ and $d$ such that 
\begin{align*}
s_{i_1}\cdots s_{i_{j}}&=[\ldots,b,a,\ldots,c,d,\ldots], \\
s_{i_1}\cdots s_{i_{j+1}}&=[\ldots,b,a,\ldots,d,c,\ldots].
\end{align*} 
And this contradicts the fact that $(c,d)$ is a leading-cell (of $f|_{\Inv(s_{i_1}\cdots s_{i{j+1}})}$, see the iterative definition of the sequences above). Indeed, we have $f(a,b)=f(c,d)$ and $\sigma^{-1}(a)<\sigma^{-1}(c)$, which contradicts the minimality of $\sigma^{-1}(c)$.
\end{itemize}
In all cases, we have a contradictions. Thus, $i_{j+1} \leq i_j$, and this concludes the proof.

\textbf{Step 2: construction of the reverse function.} Let $s_{i_1}\cdots s_{i_{\ell(\sigma)}} \in {\rm Red}(\sigma)$ and $r_1 \leq \ldots \leq r_{\ell(\sigma)}$ be a sequence of integers such that $r_j<r_{j+1}$ whenever $i_j<i_{j+1}$.
It is easy to associate a function $f:\Inv(\sigma) \to \mathbb{N}^*$ to this pair of sequences: let  $L=[(a_i,b_i)]_{1 \leq i \leq \ell(\sigma)} \in PS_{\Inv(\sigma)}(\A)$ be the sequence such that for all $j$,
\begin{align*}
s_{i_1}\cdots s_{i_{j}}&=[\ldots,a_{j},b_j,\ldots], \\
s_{i_1}\cdots s_{i_{j+1}}&=[\ldots,b_j,a_j,\ldots].
\end{align*} 
Then, we define $f:\Inv(\sigma) \to \mathbb{N}^*$ by $f(a_j,b_j)=r_j$ for all $j$. Let us first show that $f \in \SSF(\Inv(\sigma))$. In order to do so, let us consider $j<k$ such that $a_j=a_k$. We will prove that $f(a_j,b_j)<f(a_k,b_k)$, implying that $f$ is in $\SSF(\Inv(\sigma))$.
Consider the sequence $i_j,i_{j+1},\ldots,i_k$, and assume by contradiction that this sequence is decreasing. Then, for all $i \leq q <k$ we have either 
\begin{align*}
s_{i_1}\cdots s_{i_{q}}&=[\ldots,c,a_{j},\ldots], \\
\text{and} \ s_{i_1}\cdots s_{i_{q+1}}&=[\ldots,a_j,c,\ldots],
\end{align*} 
or we have
\begin{align*}
s_{i_1}\cdots s_{i_{q}}&=[\ldots,c,d,\ldots,a_{j},\ldots], \\
\text{and} \ s_{i_1}\cdots s_{i_{q+1}}&=[\ldots,d,c,\ldots,a_{j},\ldots].
\end{align*}
In other words, we obtain $s_{i_1}\cdots s_{i_{q+1}}$ from $s_{i_1}\cdots s_{i_{q}}$ either by swapping positions of $a_j$ with an integer just on its left, or by swapping positions of two integers being on the left of $a_j$. Therefore, we $a_k \neq a_j$, which is absurd. Thus, there exists $j\leq q<k$ such that $i_q \leq i_{q+1}$, but $s_{i_1}\cdots s_{i_{\ell(\sigma)}}$ is reduced, so that $i_q < i_{q+1}$. Consequently, we have $r_q<r_{q+1}$ by definition so $f(a_j,b_j)<f(a_k,b_k)$.
 
In order to complete the proof, we just have to show that this function is the inverse of $\Psi$, and this can be easily done recursively: assume by contradiction that $(a_{\ell(\sigma)},b_{\ell(\sigma)})$ is not the leading cell of $f$. Then, there exists $k<\ell(\sigma)$ such that $f(a_k,b_k)=f(a_{\ell(\sigma)},b_{\ell(\sigma)})$ and $\sigma^{-1}(a_k)<\sigma^{-1}(a_{\ell(\sigma)})$. Therefore, we have
\[ \sigma=[\ldots,a_k,\ldots,b_{\ell(\sigma)},a_{\ell(\sigma)},\ldots],\]
so that there is an integer $q$ such that $k \leq q < \ell(\sigma)$ and $i_q<i_{\ell(\sigma)}$, so that we have
\[ f(a_k,b_k)=r_k\leq r_q< r_{\ell(\sigma)}=f(a_{\ell(\sigma)},b_{\ell(\sigma)}), \]
and this is absurd. Repeating this argument, we have the expected property by induction, and this ends the proof.
\end{proof}

This construction leads to a combinatorial interpretation of $F_{\sigma}$ as a sum over a set of tableaux (which are depicted on Figure~\ref{FigPosterTab}): let us consider a permutation $\sigma \in S_n$ and denote by $A$ its inversion set $\Inv(\sigma)$ seen as as subset of boxes of $\lambda_n$. Clearly, $A$ inherits the digraph structure and the valuation of $\A$. Moreover, $A$ defines a valued digraph because $A \in IS(\A)$ (in a sense, $A$ define a ``sub-valued digraph of $\A$''), so that we can perform the peeling process on it. Obviously, the arising sequences are precisely the elements of $PS_A(\A)$, and we can represent each element of $PS_A(\A)$ as a tableau of shape $A$. That is, let $L=[z_1,\ldots,z_k]\in PS_A(\A)$, then $L$ can be represented as a tableau of shape $A$ where the box $z_i$ is filled by the integer $i$. These tableaux can be seen as the equivalent counterpart of standard tableaux within our theory. 

Similarly, we can define a family of ``semi-standard'' tableaux by the following way: let $L=[z_1,\ldots,z_k]\in PS_A(\A)$, we construct a tableau of shape $A$ by putting an integer $t_i$ in the box $z_i$, satisfying the following two conditions:
\begin{enumerate}
\item the sequence $(t_i)$ is weakly increasing along $L$;
\item a given integer cannot appear twice in the same column (equivalently, if $i<j$ and $z_i$ and $z_j$ are in the same column, then $t_i<t_j$).
\end{enumerate} 
Clearly, these tableaux are in bijection with the elements of $\SSF(A)$. Therefore, if we denote by $x^T$ the monomial $x_1^{T(1)}x_2^{T(2)}\cdots$ where $T(i)$ is the number of occurrences of $i$ in a tableau $T$ obtained by the previous method, then the Stanley symmetric function $F_{\sigma}$ is the sum over all the tableaux of these monomials $x^T$. 
\begin{figure}[!h]
\includegraphics[width=0.7\textwidth]{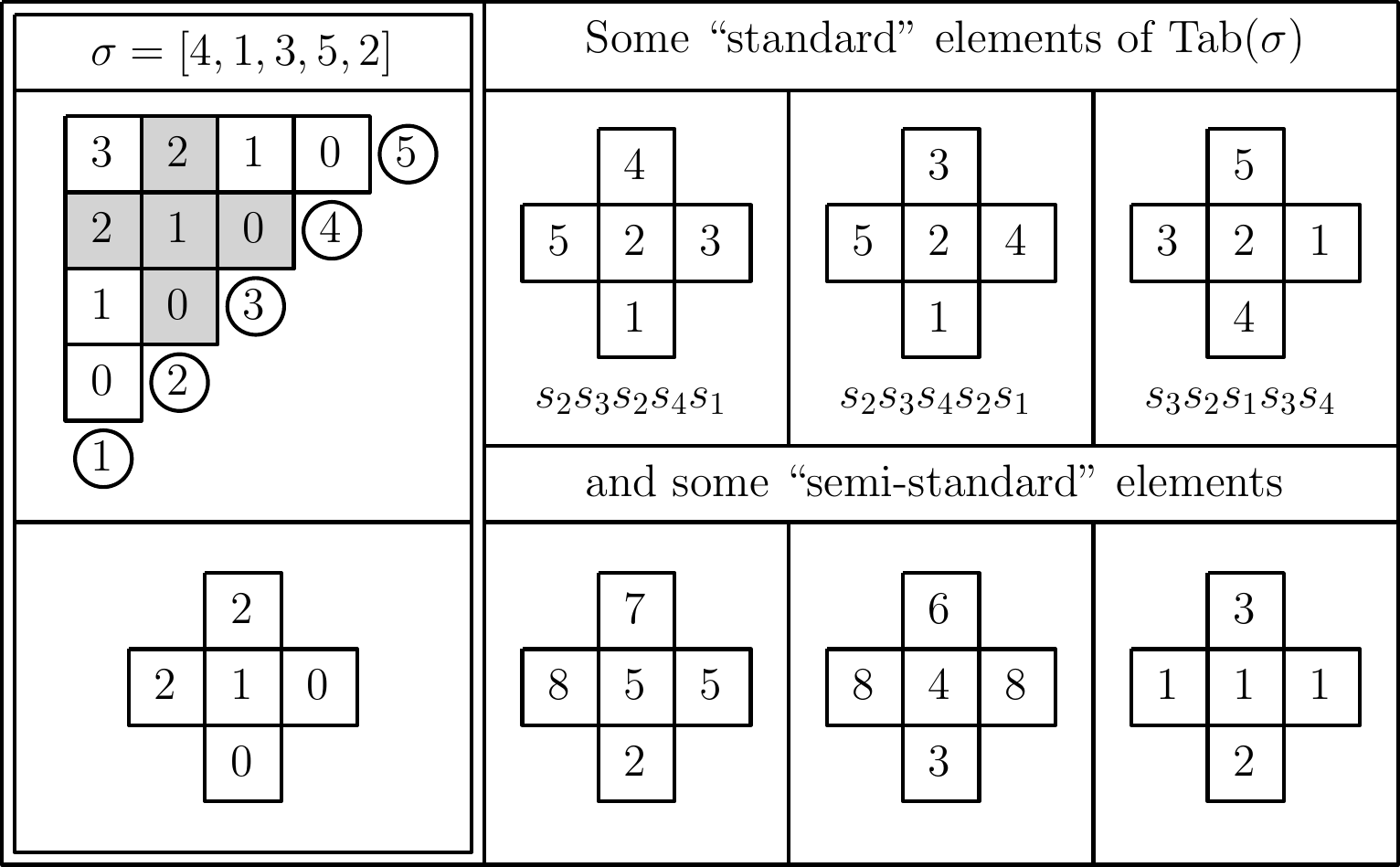}
\caption{}\label{FigPosterTab}
\end{figure}

\subsection{Type $\widetilde{A}$ and affine Stanley series}\label{SectStanAff}

In this section we apply the same method used in the previous section to the valued digraph $\widetilde{\A}=(G,\theta)$, associated with the weak order on $\widetilde{A}_n$ (see Section~\ref{SectAffineA}). Once again, the graphical representation of $G$ as a diagram leads to a natural choice for a set of generalized columns, given by the usual columns of $G$.

\begin{definition}
We denote by $\U^{col}=(\U_{(a,b)})_{1\leq a \leq n, \ a<b}$ the set of generalized columns of $\widetilde{A}$ defined by:
\[\U_{(a,b)}:=\{(a,k) \ | \ a<k \ {\rm and} \ k \not \equiv a \!\! \pmod n\}.\]
\end{definition}

As in the $A_{n-1}$ case, the series arising from this choice of generalized columns is known, namely the affine Stanley symmetric function, introduced by Lam in \cite{Lam}. Note that a combinatorial interpretation in terms of tableaux of this series has already been provided by Yun and Yoo in \cite{YY}, and the one arising from our model is very similar. Therefore, the proofs are similar. 

Let us begin with the definition of the affine Stanley symmetric function.

\begin{definition}
A sequence $(i_1,\ldots,i_k)$ of elements of $\mathbb{Z}/n\mathbb{Z}$ is called \emph{cyclically decreasing} if and only if:
\begin{itemize}
\item each element of $\mathbb{Z}/n\mathbb{Z}$ appears at most once in the sequence;
\item if there exists $p$ and $q$ such that $i_p=j$ and $i_q=j+1$, then $q<p$.
\end{itemize}
An affine permutation $\sigma \in \widetilde{A_n}$ is called cyclically decreasing if there exists a cyclically decreasing sequence $(i_1,\ldots,i_k)$ such that $s_{i_1}\cdots s_{i_k}$ is a reduced decomposition of $\sigma$ (note that $Id$ is cyclically decreasing by convention). For any $\omega\in \widetilde{A_n}$, a \emph{cyclically decreasing factorization} of $\omega$ is an expression of $\omega$ as a product $\omega=v_1v_2\cdots v_r$ such that:
\begin{itemize}
\item each $v_i \in \widetilde{A_n}$ is cyclically decreasing;
\item $\ell(\omega)=\ell(v_1)+\cdots+\ell(v_r)$.
\end{itemize}
Finally, the \emph{affine Stanley symmetric function} $F_{\omega}$ associated with $\omega$ is defined as
$$\widetilde{F}_{\omega}(x_1,x_2,\ldots):=\sum_{\omega=v_1\cdots v_r} \ x_1^{\ell(v_1)}\cdots x_r^{\ell(v_r)},$$ 
where the sum is over all cyclically decreasing factorisations of $\omega$ (see \cite{Lam}).
\end{definition}

\begin{definition}
Let $\omega \in \widetilde{A_n}$, $f \in \SSF(\Inv(\omega))$, $L=[(a_i,b_i)] \in PS_{\Inv(\omega)}(\widetilde{A})$ be a $f$-compatible peeling sequence and $s_{i_1}\ldots s_{i_{\ell(\omega)}}$ be the reduced decomposition of $\omega$ associated with $L$. We define the following factorization of $\omega$: 
\[ \Psi(f,L):=v_1v_2\cdots v_r,\]
where $v_k$ is defined as follows: for all $j$,
\begin{itemize}
\item if $f(a_j,b_j) \neq k$, then $v_k=Id$;
\item if there exists $p \leq q$ such that $f(a_j,b_j)=k$ if and only if $p \leq j \leq q$, then $v_k=s_{i_p}s_{i_{p+1}}\cdots s_{i_q}$.
\end{itemize}
\end{definition} 

\begin{proposition}\label{TechStanAff}
The function $\Psi$ does not depend on the choice of the $f$-compatible peeling sequence $L$, and $\Psi(f,L)$ is a cyclically decreasing factorization of $\omega$. 
\end{proposition}

\begin{proof}
Since $v_k$ is uniquely determined by the inversions $(a_j,b_j)$ for which $f(a_j,b_j)=k$, the first statement of the proposition is clear.

It remains to show that $v_k$ is cyclically decreasing for all $k$. If $v_k=Id$, then it is clear. Let $k$ be such that $v_k \neq Id$, hence there exists $p\leq q$ such that $v_k=s_{i_p}\cdots s_{i_q}$. If $p=q$, $v_k$ is obviously cyclically decreasing. 

We now focus on the case $p<q$. Assume by contradiction that there exists $p\leq u<v\leq q$ such that $i_u=i_v+1$. Without loss of generality, we can suppose that $v$ is minimal with this property. 
As usual, we will denote by $\omega_m$ the affine permutation $s_{i_1}\cdots s_{i_m}$. We have that $\omega_{m+1}$ is obtained from $\omega_m$ by swapping the positions of all the pair of integers $a_m+rn$ and $b_m+rn$, $r \in \mathbb{Z}$. For the sake of clarity, we will just say that the positions of $a_m$ and $b_m$ are swapped. Let us now split the study into two cases.
\begin{itemize}
\item If for all $u\leq j<v$ we have $b_j\neq a_u$, then the position of $a_u$ remains unchanged in each $\omega_j$. Thus, by minimality of $v$, $(a_v,b_v)=(a_u,b_v)$, which is absurd since $(a_u,b_u)$ and $(a_u,b_v)$ are in the same column and $f\in \SSF(\Inv(\omega),\U^{col})$.
\item If there exists $u \leq j<v$ such that $b_j=a_u$, then there exists an integer $m$ such that we have
\begin{equation*}
\omega_{j}=[\ldots,m,a_u,\ldots] \ \ \text{and} \ \
\omega_{j+1}=[\ldots,a_u,m,\ldots].
\end{equation*}
Since we have
$\omega_{u}=[\ldots,b_u,a_u,\ldots]$,
there exists $u \leq j' <j$ such that 
\begin{equation*}
\omega_{j'}=[\ldots,m,b_u,\ldots] \ \ \text{and} \ \
\omega_{j'+1}=[\ldots,b_u,m,\ldots].
\end{equation*}
Therefore, we have $f(m,b_u)=f(m,a_u)$ which is absurd since $(m,b_u)$ and $(m,a_u)$ are in the same column.
\end{itemize}
In all cases we have a contradiction, so that $v_k$ is cyclically decreasing, and this concludes the proof.
\end{proof}

We are now able to prove the main theorem of this section, using this function $\Psi$ (the proof is similar to that of the $A_{n-1}$ case).

\begin{Theorem}
For all $\omega \in \widetilde{A_n}$, we have 
\[ \widetilde{F}_{\omega}=\Gamma(\Inv_{\widetilde{A}}(\omega),\U^{col}). \]
\end{Theorem}

\begin{proof}
Thanks to Proposition~\ref{TechStanAff}, we have a map which associate to each $f\in \SSF(\Inv(\omega))$ a cyclically decreasing factorization $\omega=v_1\cdots v_r$. Moreover, we clearly have that
$$x_1^{\ell(v_1)}\cdots x_r^{\ell(v_r)}=\prod_{c \in \Inv_{\widetilde{A}}(\omega)} x_{f(c)}.$$ 
Consequently, we just have to show that this function is a bijection to prove the theorem.

Let us construct the reverse function: let $\omega=v_1\cdots v_r$ be a cyclically decreasing factorization of $\omega$. Consider any cyclically decreasing reduced decomposition of $v_i$ and concatenate them to get a reduced decomposition of $\omega$, \emph{i.e.} $\omega=(s_{i_1}\cdots s_{i_{\ell(v_1)}})(s_{i_{\ell(v_1)+1}}\cdots s_{i_{\ell(v_1)+\ell(v_2)}})\cdots$. Let $L=[(a_i,b_i)]_i \in PS_{\Inv(\omega)}(\widetilde{A})$ be the peeling sequences canonically associated with this reduced decomposition, and we define a function $f$ from $\Inv_{\widetilde{A}}(\omega)$ to $\mathbb{N}^*$ by the following way: 
\[ {\rm if} \ \ell(v_1)+\cdots+\ell(v_k)+1\leq j \leq \ell(v_1)+\cdots+\ell(v_k)+\ell(v_{k+1}), \ {\rm then} \ f(a_j,b_j)=k+1. \]
Clearly this function $f$ does not depends on the cyclically decreasing reduced decomposition chosen for each $v_i$, but depends only on the cyclically decreasing factorization of $\omega$. 

Let us prove that $f$ is in $\SSF(\Inv_{\widetilde{A}}(\omega))$
As usual, we denote $\omega_j=s_{i_1}\cdots s_{i_j}$. Set $k$ such that $\ell(v_k)\geq 2$, $p=\ell(v_1)+\cdots +\ell(v_{k-1})+1$, and $q=\ell(v_1)+\cdots +\ell(v_k)$. Now assume that there exists $p\leq u<v \leq q$ such that $f(a_u,b_u)=f(a_v,b_v)=k$ with $a_u=a_v$. Without loss of generality, we can suppose that $u$ is maximal with this property, and that $v$ is minimal with this property. Once again, there are two cases.
\begin{itemize}
\item If for all $u<j<v$, $b_j \neq a_u$, then the position of $a_u$ remains unchanged in $\omega_{u+1},\ldots,\omega_{v-1}$. Hence by minimality of $v$ we have $i_v=i_u+1$, which is absurd since $v_k$ is cyclically decreasing.
\item If there exists $u<j<v$ such that $b_j=a_u$, then there exists $u<j'<j$ such that $(a_{j'},b_{j'})=(a_j,b_u)$ since $b_u$ is just on the left of $a_u$ in $\omega_u$. Hence we found $j'>u$ such that there exists $j>j'$ with $a_{j'}=a_j$, which is absurd by maximality of $u$. 
\end{itemize}
Hence $f \in \SSF(\Inv(\omega))$, and this function clearly invert the one defined earlier, and this achieves the proof of the theorem.
\end{proof}

\end{document}